\documentclass[11pt,oneside]{amsart}
\usepackage[margin=1.4in]{geometry}

\pdfoutput=1

\usepackage{amsmath, amscd}
\usepackage{amsfonts}
\usepackage{amssymb}
\usepackage{amsthm}
\usepackage{upref}
\usepackage{multirow}
\usepackage{graphicx}         
\usepackage{subfig}            
\usepackage{url}               
\usepackage{mathtools}
\usepackage{pinlabel}
\usepackage{color}
\usepackage{hyperref}
\usepackage{changepage}
\usepackage{epstopdf}
\usepackage{enumerate}

\newcommand\quotient[2]{\raise1ex\hbox{$#1$}\Big/\lower1ex\hbox{$#2$}}
\newcommand{\gspan}[1]{\left\langle{#1}\right\rangle}

\newcommand{\Z}{\mathbb{Z}}

\newcommand{\Out}{\operatorname{Out}}

\newcommand{\Aut}{\operatorname{Aut}}

\newcommand{\st}{\ensuremath{\operatorname{st}}}

\newcommand{\SL}{\operatorname{SL}}
\newcommand{\GL}{\operatorname{GL}}
\newcommand{\IA}{\operatorname{IA}}
\newcommand{\Tr}{\operatorname{Tr}}

\newcommand{\abs}[1]{\left| #1 \right|}

\newcommand{\lk}{\ensuremath{\operatorname{lk}}}
    
   \newcommand{\leinf}{\leq_\infty}
   \newcommand{\leinfhat}{\hat{\leq}_\infty}
   \newcommand{\lehat}{\hat{\leq}}
   \newcommand{\lehatcomp}{\ \hat{\leq}\ }
      \newcommand{\nlehatcomp}{\ \hat{\not\leq}\ }

\newtheorem{thm}{Theorem}[section]
\newtheorem{thmspecial}{Theorem}
\newtheorem{corspecial}[thmspecial]{Corollary}

\newtheorem{cor}[thm]{Corollary}
\newtheorem{lem}[thm]{Lemma}
\theoremstyle{remark}
\newtheorem{rem}[thm]{Remark}
\theoremstyle{definition}
\newtheorem{defn}[thm]{Definition}
\theoremstyle{plain}
\newtheorem{prop}[thm]{Proposition}
\theoremstyle{remark}

\theoremstyle{definition}

\theoremstyle{definition}

\theoremstyle{plain}

\theoremstyle{definition}

\newcounter{claimcounter}[thm]
\numberwithin{claimcounter}{thm}
\newcounter{casecounter}

\newenvironment{case}{\stepcounter{casecounter}{\noindent\textit{Case \thecasecounter.}}}{\vspace{0.1cm}}
\newcounter{subcasecounter}[casecounter]
\numberwithin{subcasecounter}{casecounter}

\newcommand{\mids}{\, | \, }
\newcommand{\m}{\ensuremath{^{-1}}}
\newcommand{\ad}{\operatorname{ad}}
\renewcommand{\phi}{\varphi}
\newcommand{\cal}[1]{\mathcal{#1}}
\newcommand{\grp}[1]{\langle #1 \rangle}
\newcommand{\N}{\mathbb{N}}
\newcommand{\G}{\cal{G}}
\newcommand{\Ghat}{\hat{\cal{G}}}
\newcommand{\Autfin}{\Aut^{\operatorname{fin}}}

\renewcommand{\Im}{\operatorname{Im}}

\title[Outer automorphism groups of graph products]{Outer automorphism groups of graph products: subgroups and quotients}

\author{Andrew Sale}
\author{Tim Susse}

\begin{document}
    \begin{abstract}
    We show that the outer automorphism groups of graph products of finitely generated abelian groups satisfy the Tits alternative, are residually finite, their so-called Torelli subgroups are finitely generated, and they satisfy a dichotomy between being virtually nilpotent and containing a non-abelian free subgroup that is determined by a graphical condition on the underlying labelled graph.
    	
	Graph products of finitely generated abelian groups simultaneously generalize right-angled Artin groups (RAAGs) and right-angled Coxter groups (RACGs), providing a common framework for studying these groups. Our results extend a number of known results for the outer automorphism groups of RAAGs and/or RACGs by a variety of authors, including Caprace, Charney, Day, Ferov, Guirardel, Horbez, Minasyan, Vogtmann, Wade, and the current authors.
    \end{abstract}

\maketitle

 A graph product is a group construction introduced by Green \cite{Green:graphprod} that generalizes free products and direct products.
Given a simplicial graph $\Gamma$ with vertex set $V$, and a collection of 
groups $\{G_v\}_{v\in V}$, referred to as the vertex groups, we can form the \emph{graph product} of these groups over $\Gamma$
by taking the group generated by all vertex groups and declaring that every element of $G_v$ commutes with every element of $G_w$ if and only if $v$ and $w$ are adjacent in $\Gamma$. Graph products of finitely generated abelian groups generalize the classes of right-angled Artin and Coxeter groups  (RAAGs and RACGs respectively), and many properties of of RAAGs and RACGs are shared by every group in this class. Recall that a RAAG is a graph product where $G_v\cong \Z$ for all $v\in V$, and a RACG is a graph product where $G_v\cong \Z_2$ for all $v\in V$.

Recently, there has been a significant amount of research into the (outer) automorphism groups of RAAGs, for example in \cite{CCV_automorphisms, CV09, CV11, Day_symplectic, Day_solvable, DayWade, DuncanRemeslennikovII, GuirardelSale-vastness, Minasyan:RFRAAG, WadeThesis}. 
Some attention has been paid to the structure of outer automorphism groups of general graph products of abelian groups, focusing mainly on the case when all vertex groups have finite order (see for example \cite{COrredor-Gutierrez, CRSV_no_SIL, GPR_automorphisms, SaleSusse}). 
In this paper, we extend the techniques and results of many of these papers to the general framework of automorphisms of graph products of finitely generated abelian groups.
When vertex groups are not abelian, Genevois and Martin have found a finite generating set for some outer automorphism groups of graph products, and studied other geometric properties (e.g.~acylindrical hyperbolicity) \cite{GenevoisMartin}, while Ferov has investigated Grossman's Property (A) and residual properties \cite{Ferov:RF}.

One of the properties we investigate is whether our outer automorphism groups satisfy the Tits Alternative.
Recall that a group $G$ is said to satisfy the \emph{Tits Alternative} if for every subgroup $H$ of $G$ either
 $H$ contains a non-abelian free subgroup, or
 $H$ is virtually solvable.

The Tits Alternative holds for linear groups \cite{TitsAlt}, 
for hyperbolic groups \cite{Gromov:hyperbolic, GhysDLH:hyperbolic},
for mapping class groups \cite{Ivanov:Tits, McCarthy:Tits},
for many groups acting on $\operatorname{CAT}(0)$ cube complexes \cite{SageevWise},
and for $\Out(F_n)$ \cite{BFH:Tits1}. 
 Except for linear groups, in each of the above, solvable subgroups are in fact virtually abelian and for hyperbolic groups, even virtually cyclic. 
The Tits Alternative is closed under taking graph products \cite{AntolinMinasyanTA}, and so any graph product of finitely generated abelian groups also satisfies it.

The Tits Alternative is known to hold for outer automorphism groups of RAAGs. This was initiated by Charney and Vogtmann \cite{CV11} and completed by Horbez \cite{HorbezTA}, who showed in particular that the Tits Alternative holds for outer automorphism groups of free products (quoted below as Theorem~\ref{thm:HorbezFP}).

Here we extend these results to \emph{all} graph products of finitely generated abelian groups.

\begin{thmspecial}\label{thms:OutGPTA}
	Let $\G$ be a graph product of finitely generated abelian groups.
	
	Then $\Out(\G)$ satisfies the Tits Alternative.
\end{thmspecial}

A special case of Theorem~\ref{thms:OutGPTA} is the following new result.

\begin{corspecial}
	The outer automorphism group of a right-angled Coxeter group satisfies the Tits Alternative.
\end{corspecial}

We remark that we show a stronger version of the Tits alternative for Theorem~\ref{thms:OutGPTA}, where the virtually solvable subgroups are virtually polycyclic (see Corollary~\ref{cor:polycyclic TA}).

In the study of the Tits Alternative for the outer automorphism group of a RAAG, Horbez' result for free products  allows us in particular to reduce to the case when $\Gamma$ is connected. 
Charney and Vogtmann used a pre-order on the vertices of $\Gamma$, which was introduced by Servatius \cite{Servatius} and is sometimes called domination, to apply an inductive method to prove the Tits Alternative for certain graphs $\Gamma$.
When paired with Horbez' theorem, Charney and Vogtmann's inductive method produces the result for all RAAGs.

To prove Theorem~\ref{thms:OutGPTA} we generalize the method of Charney and Vogtmann.
One part of this is extending the definition of the domination relation of Servatius. 
In particular, we define two pre-orders, inspired by domination and which we denote $\leq$ and $\leq_\infty$. The first of these is used for the induction process, while the latter is used for other results in the paper.

The next property we look at concerns the finite quotients of the graph product.
Recall that a group $G$ is said to be \emph{residually finite} if for every $g\in G$ there is a finite group $F$ and a homomorphism $\phi\colon G\to F$ so that $\phi(g)\neq 1$. While Baumslag proved that $\Aut(G)$ is residually finite whenever $G$ is \cite{Baumslag:AutRF}, this result does not extend immediately to $\Out(G)$. 
Using the techniques developed to prove Theorem~\ref{thms:OutGPTA} we obtain the following, previously proven by Ferov in a more general setting \cite{Ferov:RF} using very different techniques.

\begin{thmspecial}\label{thms:OutGPRF}
	Let $\G$ be a graph product of finitely generated abelian groups.

	Then $\Out(\G)$ is residually finite.	
\end{thmspecial}

This generalizes results of Charney and Vogtmann \cite{CV11} and Minasyan \cite{Minasyan:RFRAAG} for RAAGs, and Carprace and Minasyan \cite{CapraceMinasyan:RFRACG} for RACGs (which was generalized by Carette to all Coxeter groups \cite{Carette}).

For a graph product $\G$ of finitely generated abelian groups,
we also study the structure of the subgroup of the (outer) automorphism group which acts as the identity on the abelianization of $\G$, denoted $\IA_\G$. 
This is sometimes called the Torelli subgroup, after the corresponding subgroup of the mapping class group. 
For mapping class groups, this group is known to be finitely generated by results of Birman and Powell \cite{Birman:Torelli, Powell:Torelli}, while for $\Out(F_n)$ this is a classic result of Magnus \cite{Magnus:IAfree}. Day extended the result of Magnus to outer automorphism groups of RAAGs \cite{Day_symplectic} (later proved independently by Wade \cite{WadeThesis}).
We extend the result further, proving the following theorem.
An explicit generating set is given in Theorem~\ref{thm:Torelli gen set}.

\begin{thmspecial}\label{thms:Torelli fg}
	Let $\G$ be a graph product of finitely generated abelian groups.
	
	Then the Torelli subgroup $\IA_\G$ is finitely generated.
\end{thmspecial}

Our proof generally follows the structure of both Day and Wade's arguments, which in turn follow Magnus. In the case where the graph product $\G$ has a vertex with finite order elements, complications arise from certain finite order automorphisms. We introduce a new finite index subgroup of $\Aut(\G)$, which does not contains these automorphisms, and show that $\IA_\G$ is contained in this subgroup.

We then use this Theorem to prove a broad structural result for $\Out(\G)$, extending a theorem of Day for RAAGs \cite{Day_solvable} and the authors in the case of graph products of finite abelian groups \cite{SaleSusse}. In general such a graph product is uniquely determined by a labeled graph $(\Gamma,o)$, where $o$ is a function that assigns to each vertex of $\Gamma$ a power of a prime number or $\infty$.
The group $\G(\Gamma,o)$ is generated by taking the vertex group $G_v$ to be either $\Z$ if $o(v) =\infty$, or $\Z_{o(v)}$ otherwise.
Definitions of the other relevant terms appear in Sections~\ref{sec:bkgdauto} and \ref{sec:preorder}.

\begin{thmspecial}\label{thms:no SIL nil}
	Let $(\Gamma,o)$ be a finite labeled graph and $\G = \G(\Gamma,o)$.
	Then $\Out(\G)$ contains a nonabelian free subgroup if and only if $(\Gamma,o)$ contains either
	\begin{itemize}
		\item a $\leq_\infty$--equivalence class of size at least 2;
		\item a non-Coxeter SIL;
		\item a STIL;
		\item an FSIL.
	\end{itemize}
	Otherwise, $\Out(\G)$ contains a finite index subgroup that is nilpotent.
\end{thmspecial}

We note that the nilpotency class is also determined. See Theorem~\ref{thm:no SIL nil class}.

In light of Theorem~\ref{thms:no SIL nil} we like to think of the condition that $(\Gamma,o)$ has no non-Coxeter SIL, no STIL, and no FSIL as describing a situation where the labeled graph has ``no free SIL.''
To prove this theorem, we show that the structure of a group with ``no free SIL'' is very rigid, providing a short-exact sequence for a finite-index subgroup of $\Out(\G)$, generalizing \cite[Theorem 2 (2)]{GuirardelSale-vastness}.

\begin{thmspecial}\label{thms:no SIL ses}
	Let $(\Gamma,o)$ be a finite labeled graph and $\G = \G(\Gamma,o)$.
	There is a finite-index subgroup $\cal{O}$ of $\Out(\G)$ such that
	if $(\Gamma,o)$ does not contain either
		\begin{itemize}
			\item a non-Coxeter SIL,
			\item a STIL,
			\item an FSIL,
		\end{itemize}
	then there is a short exact sequence
	$$1\to P \to \cal{O} \to \prod_{i=1}^k \SL(n_i,\Z) \to 1$$
	where $P$ is finitely generated and virtually nilpotent, and $n_1 , \ldots, n_k$ are the sizes of the $\leq_\infty$--equivalence classes containing more than one element.
\end{thmspecial}

Once again, the nilpotency class of a finite-index subgroup of $P$ is determined, given in Theorem~\ref{thm:no SIL SES} below.

In particular, we immediately deduce the following.

\begin{corspecial}
	Suppose $(\Gamma,o)$ does not contain either
	\begin{itemize}
		\item a non-Coxeter SIL,
		\item a STIL,
		\item an FSIL.
	\end{itemize}
	Then $\Out(\G)$ is large if and only if $(\Gamma,o)$ contains a $\leq_\infty$--equivalence class of size precisely 2.
	
\end{corspecial}

This paper is structured as follows.
In Section~\ref{sec:prelim} we will provide background on graph products and their automorphism groups and define of the key concepts which we will use throughout the remainder of the paper, specifically our notions of domination and equivalence of vertices and restriction and factor maps. In Section~\ref{sec:fi subgroups} we describe several subgroups of $\Out(\G)$, and prove that each is finite index. In Sections~\ref{sec:amalg proj} and~\ref{sec:TARF} we prove Theorems~\ref{thms:OutGPTA} and~\ref{thms:OutGPRF} on the Tits alternative and residual finiteness. In Section~\ref{sec:torelli} we prove Theorem~\ref{thms:Torelli fg} and use that theorem in Section~\ref{sec:noSIL} to prove Theorem~\ref{thms:no SIL nil}.

\medskip 

\noindent
{\bf Acknowledgements:} The authors wish to thank Michal Ferov and Ashot Minasyan for bringing to our attention past results on residual finiteness.

\section{Preliminaries}\label{sec:prelim}

\subsection{Graphs}
Given two graphs $\Gamma_1$ and $\Gamma_2$, their \emph{join}, is the graph $\Gamma_1\star\Gamma_2$ obtained from the disjoint union of $\Gamma_1$ and $\Gamma_2$ by adding an edge between every vertex of $\Gamma_1$ and every vertex of $\Gamma_2$.

We say that $\Lambda$ is an \emph{induced subgraph} of $\Gamma$ if whenever $v,w$ are vertices of $\Lambda$, they are adjacent in $\Lambda$ if and only if they are adjacent in $\Gamma$.

Given a subset $X$ of vertices of $\Gamma$, the \emph{link} of $X$ is the subgraph $\lk_\Gamma X$ induced by the vertices $v$ that are adjacent to every vertex in $X$.

The \emph{star} of $X$, denoted $\st_\Gamma X$, is the join $X \star  \lk_\Gamma X$ (which is, by definition of link, an induced subgraph of $\Gamma$).

When there is no ambiguity over what is the underlying graph $\Gamma$, we will omit $\Gamma$ from the notations and write $\lk X$ and $\st X$ instead.

\subsection{Graph products}\label{sec:graph products}

Let $\Gamma$ be a simplicial graph with vertex set $V$.
Let $\{G_v\}_{v\in V}$ be a set of groups indexed by $V$.
The graph product $\cal{G}(\Gamma,\{G_v\})$ is the quotient of the free product of the groups $G_v$ by the relations obtained by saying $G_v$ and $G_u$ commute when $v$ and $u$ are adjacent in $\Gamma$.

Suppose each $G_v$ is finitely generated abelian.
Then $\cal{G}(\Gamma,\{G_v\})$ is isomorphic to a graph product $\cal{G}(\Lambda,\{H_v\})$ where each vertex group $H_v$ is either infinite cyclic, or finite cyclic of prime-power order.
Thus we may assume $\cal{G}(\Gamma,\{G_v\})$ is already given by such a presentation.

\begin{defn}
	Let $\Gamma$ be a graph with vertex set $V$, and let $o : V \to P \cup \{\infty\}$ be a map, where $P\subset\N$ is the set of prime powers.
	Define $\cal{G}(\Gamma,o)$ to be the group with generating set $V$ and relators as follows:
	two generators $u,v$ commute whenever they are adjacent;
	the order of  a generator $v$ is $o(v)$.
Except where we need to be explicit, we omit the graph and its labelling from the notation, writing $\G$ instead of $\G(\Gamma,o)$.
	It has presentation
	\begin{equation*}
	\G = \G(\Gamma,o) = \grp{ V \mid [u,v]=1 \textrm{ if $u,v$ adjacent in $\Gamma$; } w^{o(w)}=1 \ \forall w\in V, o(w) < \infty}.
	\end{equation*}
\end{defn} 

If $o(v)$ is the order of $G_v$ (which is still assumed to be cyclic), then $\cal{G}(\Gamma,o)$ is isomorphic to $\cal{G}(\Gamma,\{G_v\})$.

Finally, we describe a natural family of subgroups of $\G$.

\begin{defn} 
Given an induced subgraph $\Lambda$ in $\Gamma$, we define $\cal{G}_\Lambda$ to be the \emph{special subgroup} generated by vertices in $\Lambda$.
Note that $\cal{G}_\Lambda$ is isomorphic to $\cal{G}(\Lambda , o | _\Lambda)$, where $o | _\Lambda$ denotes the restriction of the order map to the vertex set of $\Lambda$.
\end{defn}

\subsection{Automorphisms}\label{sec:bkgdauto}

Corredor and Gutierrez generalized the Laurence--Servatius generators for the RAAG case to describe a finite generating set for $\Out(\G)$ \cite{COrredor-Gutierrez}.
Their set consists of four types of automorphism.

\emph{Labeled graph automorphisms.}
Any automorphism of $\Gamma$ that preserves the labeling $o$ gives rise to an automorphism of $\G$.

\emph{Factor automorphisms.}
Any automorphism of a vertex group can be extended by the identity on all other vertex groups to give an automorphism of $\G$.

\emph{Transvections.}
Let $u,v \in V$. Suppose that either:
\begin{enumerate}
	\item $\lk(u) \subseteq \st(v)$ and $o(u) = \infty$; or
	\item $\st(u) \subseteq \st(v)$ and $o(v)$ divides $o(u)$.
\end{enumerate}
Then we can define automorphisms of $\G$ called transvections by $R_u^{v^k} \colon u \mapsto uv^k$ and $L_u^{v^k} \colon u \mapsto v^ku$ (fixing all other vertices), where $k=1$ if $o(u)=\infty$ and if
$o(u)=p^i$, $o(v) = p^j$, then $k = \max\{p^{j-i},1\}$.
When we refer to a transvection $R_u^{v^k}$ we will, without saying, take $k$ to be the value above, unless otherwise specified.

\emph{Partial conjugations.}
Fix a vertex $v$ and $C$ a union of connected components of $\Gamma \setminus \st(v)$. Define the partial conjugation $\pi^v_C$ by sending each $z\in C$ to $z^v = vzv\m$, and fixing all other vertices.

Given a transvection $R^{x^k}_y$ or a partial conjugation $\pi^x_C$, we call $x$ the \emph{multiplier} of the automorphism and $y$ or $C$ the support.

\begin{thm}[Corredor--Gutierrez \cite{COrredor-Gutierrez}]
	\label{thm:CGaut}The set of all labeled graph automorphisms, factor automorphisms, transvections $R_u^{v^k}$ for $u,v$ that satisfy either (1) or (2) above, and partial conjugations is a finite generating set for $\Aut(\G)$.
\end{thm}

We will consider several finite index subgroups of $\Aut(\G)$ in Section~\ref{sec:fi subgroups}, generated by subsets of the above automorphisms. We will be interested mainly in transvections and partial conjugations.

\begin{defn}
Let $\Aut^1(\G)$ be the subgroup of $\Aut(\G)$ generated by the set of all transvections and all partial conjugations.
	
Let $\Aut^0(\G)$ be the subgroup of $\Aut(\G)$ generated by $\Aut^1(\G)$ and all factor automorphisms.
	
Let $\Out^1(\G)$ and $\Out^0(\G)$ be the images of $\Aut^1(\G)$ and $\Aut^0(\G)$ in $\Out(\G)$, respectively.
\end{defn}

Below, Proposition~\ref{prop:finite index subgroups} shows that $\Aut^1(\G)$ and $\Aut^0(\G)$ have finite index in $\Aut(\G)$. Studying $\Aut^1(\G)$, the following relations will be useful, verification is left to the reader.

\begin{lem}\label{lem:relators}
	Let $x,y,z,w$ be distinct vertices in $\Gamma$. When the transvections (or their powers) exist, we get the following relations:
	\begin{equation}
	\label{eq:transv rel1}
	R^{x^k}_yR^{y^j}_z = R^{y^j}_z R^{x^{kj}}_z R^{x^k}_y, 
	\end{equation}
	\begin{equation}
	\label{eq:transv rel2}
	R^{x^k}_yR^{x^j}_z = R^{x^j}_zR^{x^k}_y, 
	\end{equation}
	\begin{equation}
	\label{eq:transv commute}
	R^{x^k}_y R^{w^j}_z = R^{w^j}_z  R^{x^k}_y.
	\end{equation}
\end{lem}

The notion of a SIL was introducted in \cite{GPR_automorphisms}, and extended to the notion of a STIL and an FSIL in \cite{SaleSusse}. These graphical characteristics have been useful in understanding the structure of $\Out(\G)$ (see \emph{e.g.} \cite{Day_solvable, GPR_automorphisms, GuirardelSale-vastness, SaleSusse}).

\begin{defn} Let $(\Gamma,o)$ be a graph with $\G_\Gamma$ the corresponding graph product, and let $x,y,z,w$ be vertices. We say:
	\begin{itemize}
		\item $(x,y\mids w)$ is a \emph{SIL (Separating Intersection of Links)} if $x$ and $y$ are not adjacent and $w$ is contained in a component of $\Gamma\setminus\left(\lk(x)\cap\lk(y)\right)$ which does not contain $x,y$;
		\item $(x,y\mids w)$ is a \emph{non-Coxeter SIL} if $(x,y\mids w)$ is a SIL and either $o(x)\neq 2$ or $o(y)\neq 2$;
		\item $(x,y,z\mids w)$ is a \emph{STIL (Separating Triple Intersection of Links)} if $\gspan{x,y,z}\le \G_\Gamma$ is not virtually abelian and $w$ is contained in a component of $\Gamma\setminus\left(\lk(x)\cap\lk(y)\cap\lk(z)\right)$ which contains none of the vertices $x,y,z$;
		\item $\{x,y,z\}$ is an \emph{FSIL (Flexible SIL)} if $(x,y\mids z), (y,z\mids x),$ and $(z,x\mids y)$ are all SILs.
	\end{itemize}
\end{defn}

The following is proved in \cite[Section 4]{GPR_automorphisms} and describes precisely when partial conjugations do not commute.

\begin{lem}\label{lem:partial conj not commute}
	Suppose $x,y$ are distinct vertices. Two partial conjugations $\pi^x_C,\pi^y_D$ in $\Out(\G)$ do not commute if and only if
	there is a SIL $(x,y\mids z)$ and either
	\begin{itemize}
		\item $z\in C=D$,
		\item $x\in D$ and $z\in C$,
		\item $y\in C$ and $z\in D$,
		\item $x\in D$ and $y\in C$.
	\end{itemize}
\end{lem}

The authors introduced STILs, FSILs, and non-Coxeter SILs in \cite{SaleSusse} to extend the result of \cite{GPR_automorphisms} to give a precise description of when $\Out(\G)$ is virtually abelian in case that $\G$ is a graph product of \emph{finite} abelian groups.

The following describes how SILs can overlap to give a STIL.

\begin{lem}[{\cite[Lemma 1.7]{SaleSusse}}]\label{lem:overlapping SILs give STIL}
	Suppose $x,y,z,w$ are distinct vertices and $(x,y\mids w)$ and $(y,z\mids w)$ are SILs.
	Then $(x,y,z\mids w)$ is a STIL.
\end{lem}

\subsection{Pre-orders on the vertices of $\Gamma$}\label{sec:preorder}
In the study of RAAGs, the pre-order on the vertices of $\Gamma$ given by $u\leq v$ if and only if $\lk(u) \subseteq \st(v)$ is invaluable. 
It determines precisely when the transvection $R_u^v$ is an automorphism.
We extend this definition to graph products of abelian groups, and define a second one that will be useful when we study finite index subgroups of $\Aut(\G)$.

\begin{defn}
	For vertices $u,v$ of $\Gamma$, say $u \leq v$ if and only if $R_u^{v^k} \in \Aut(\G)$ for some positive integer $k$.
	If also $o(u) = \infty$, then write $u\leinf v$.
	
	Equivalence classes of the pre-orders $\leq$ and $\leinf$ are denoted $[v]$ and $[v]_\infty$ respectively.
	We will refer to them as either $\leq$--equivalence classes, or $\leq_\infty$--equivalence classes. The phrase ``equivalence class'' on its own will also be used to mean a $\leq$--equivalence class.
\end{defn}

We will consider further subgroups of $\Aut(\G)$, by restricting the types of transvections in our generating set. 

\begin{defn}Let $\Aut^1_\infty(\G)$ be the subgroup of $\Aut(\G)$ generated by partial conjugations and transvections $R^v_u$ with $u\le_\infty v$. 
	
Further, let $\Aut^0_\infty(\G)$ but the subgroup generated by $\Aut^1_\infty(\G)$ and all factor automorphisms. 

Let $\Out^1_\infty(\G)$ and $\Out^0_\infty(\G)$ be the corresponding subgroups of $\Out(\G)$.\end{defn}

Since $\lk(u) \subseteq \st(v)$ is required for either $u\leq v$ or $u\leinf v$, 
many properties of $\leq$ that are familiar from the RAAG situation carry through to our situation, and also hold for $\leinf$. The following Lemmas record these similarities (c.f. \cite{CV09, CV11} for the RAAG case).

\begin{lem}\label{lem:eq classes infinite or finite}
	Let $X$ be a $\leq$--equivalence class.
	Then the special subgroup $\G_X$ is isomorphic to either a non-abelian free group, a free abelian group, or a finite abelian $p$--group for some prime $p$.
\end{lem}

\begin{proof}
	If $X$ contains only one vertex, the result is obvious. So suppose $X$ contains at least two vertices, $u$ and $v$.
	If $o(u) = \infty$, since $v\leq u$, the transvection $R_v^u$ must exist. However, this is only possible if $v$ is also of infinite order.
	Hence either every vertex of $X$ is of infinite order, or every vertex has finite order.
	If all vertices have finite order, then existence of transvections $R^{u^k}_v$ imply that there is a prime $p$ such that each vertex in the equivalence class has order that is a power of $p$.
	 
	Now, if $u,v$ are adjacent, and $w$ is a third vertex in $X$, then $v\in \lk(u)\subseteq \st(w)$. Thus, $X$ is either a clique or has no edges. Thus, $\G_X$ is either abelian or a free product. Further, if $v\in X$ with $o(v)<\infty$, then for all $w\in X$, $o(w)<\infty$ and $\st(v)=\st(w)$. Thus, when $\G_X$ consists of finite order vertices, then $\G_X$ is abelian, and is a finite abelian $p$--group. The statement of the lemma follows.
\end{proof}

In view of Lemma~\ref{lem:eq classes infinite or finite}, we can define the notion of an \emph{abelian equivalence class}, when it generates an abelian special subgroup, or a \emph{free equivalence class} otherwise. We can also distinguish between \emph{infinite} and \emph{finite} equiavlance classes, depending on whether the vertices in the equivalence class are infinite or finite order, respectively.

\begin{lem}\label{lem:infty-stars cover Gamma}
	Suppose $\Gamma$ is connected and is not equal to the star of a vertex.
	\begin{enumerate}
		\item Every vertex $v$ in $\Gamma$ contains a maximal equivalence class in its link.
		\item If $u,v$ are non-adjacent vertices in $\Gamma$ such that $\lk(u)\cap \lk(v)$ is non-empty, then $\lk(u)\cap\lk(v)$ contains a maximal vertex.
	\end{enumerate}
\end{lem}

\begin{proof}
	Note that it is enough to prove the second statement of the lemma. Fix a vertex $v$.
	Since $\Gamma$ is connected and not equal to the star of $v$ there is some vertex $u$ that is distance $2$ from $v$.
	Let $w \in \lk(u)\cap \lk(v)$ be maximal in this set, and suppose $z\geq w$.
	Then $u,v\in \lk (w) \subseteq \st(z)$.
	Since $u,v$ are not adjacent, $z$ is distinct from both and hence $z\in\lk(u)\cap\lk(v)$. Now, by our choice of $w$, it must be equivalent to $z$.
	Thus, $w$ is maximal in $\Gamma$.
\end{proof}

\begin{lem}\label{lem:maximal eq classes not star separated}
	Let $X$ be a $\leq$--maximal equivalence class or a $\leq_\infty$--maximal equivalence class.
	
	Then for every $v\in \Gamma\setminus X$, either $X\subseteq \st(v)$, or $X$ is contained in one connected component of $\Gamma \setminus \st(v)$.
\end{lem}

\begin{proof}
	First suppose $X$ is an equivalence class whose link is empty. 
	This is possible only if $X$ is either a complete subgraph which is also a connected component of $\Gamma$, or $X$ consists of more than one isolated vertex of infinite order.
	In the first case, the statement of the Lemma is obvious.
	So, assume that $X$ is a set of at least two isolated, infinite order, vertices.
	If $X=\Gamma$ then the Lemma is vacuous, so we may assume there is a vertex $y$ that is not in $X$.
	We will then have $x\leq y$ for $x\in X$, since $\lk(x)$ is empty, a contradiction.
	
	Hence a $\leq$--maximal, or $\leq_\infty$--maximal equivalence class of $\Gamma$ is either equal to $\Gamma$, equal to a component of $\Gamma$, or has nonempty link, with the first two cases resolved.
	So suppose that $\lk X\neq\emptyset$. We show that if $X \cap \st(v)$ is non-empty, then $X\subseteq \st(v)$. Indeed, if $x\in X\cap \st(v)$ and $y\in X$, then since $x\neq v$ it follows that $v\in\lk(x)\subseteq \st(y)$. Thus, $y\in \st(v)$ as well.
	
	Now, assume $X\cap \st(v)$ is empty.
	If $X$ is abelian, then it is immediate that it is contained in one connected component of $\Gamma \setminus \st(v)$.
	Otherwise, $X$ is a free equivalence class, and thus by Lemma~\ref{lem:eq classes infinite or finite}, $o(x)=\infty$ for all $x\in X$. Thus, since $X$ is maximal, we cannot have $\lk X \subseteq \st(v)$.
	Since we may assume the link of $X$ is nonempty, there is some $y\in \lk X$ that is not in $\st(v)$, and this gives a path between any two vertices in $X$ that avoids $\st(v)$.
\end{proof}

\begin{lem}\label{lem:preservation of maximal stars}
	Let $X$ be a $\leq$--maximal equivalence class (respectively a $\leq_\infty$--maximal equivalence class), 
	and let $\Phi\in\Out^0(\G)$ (respectively $\Phi \in \Out^0_\infty(\G)$).
	
	Then there exists a representative $\phi$ of $\Phi$ so that
	$\phi(\gspan{X})=\gspan{X}$ and $\phi(\gspan{\st X})=\gspan{\st X}$.
\end{lem}

\begin{proof} Let $\Phi\in\Out^0(\G)$. It is enough to prove the statement when $\Phi$ is either a partial conjugation or a transvection $R_u^{v^k}$ with $u\le v$, since the statement is clear for factor automorphisms. 
	
	We first consider $X$, a $\le$--equivalence class. 
	
	Suppose first that $\pi=\pi^v_C$ is a partial conjugation. If $v\not\in\st X$, then by Lemma~\ref{lem:maximal eq classes not star separated}, $X$ is contained in a single connected component, call it $C'$, of $\Gamma\setminus\st(v)$. Thus, there exists a representative of $\pi$ which acts trivially on $\gspan{X}$. Now, since $d(X,v)\ge 2$, $\st X$ is contained in $C'\cup\st(v)$. Thus, this same representative acts trivially on $\gspan{\st X}$
	
	If instead $v\in \lk X$, then for any representative $\phi$ of $\Phi$ with $\phi(v)=v$, we have that $\phi$ preserves $\gspan{\st(X)}$.
	
	Suppose now that $\Phi=R^{v^k}_u$, for some vertex $u\in\st X$. If $u\in X$, then since $u\le v$, and $X$ is maximal, we must have that $v\in X$ as well. In this case, $\Phi$ has a representative that acts preserves $\gspan{X}$ and trivially on $\gspan{\lk X}$.
	
	If $u\in\lk X$, then since $X\subseteq \lk(u)\subseteq\st(v)$, we must have that $v\in\st X$ as well. Thus, $\Phi$ has a representative which acts trivially on $\gspan X$ and preserves $\gspan{\lk X}$, and we reach the conclusion.
	
	The same argument holds if $X$ is a $\le_\infty$--equivalence class and $\Phi\in\Out^0_\infty(\G).$

\end{proof}

\subsection{Restriction and factor maps}\label{subsec:ResFact}

In order to prove results for graph products, a standard strategy is to induct the number of vertices of the underlying graph. We record here a family of homomorphisms to smaller graphs that we appeal to. For more detail on these maps, consult \cite{CV09, DayWade}

We begin with \emph{restriction maps}.
Given a special subgroup $\G_\Lambda$ of $\G$, generated by an induced subgraph $\Lambda$ of $\Gamma$,
we may try to restrict automorphisms of $\G$ to $\G_\Lambda$.
Whenever we have an outer automorphism that preserves $\G_\Lambda$ up to conjugacy, we can compose the restriction with an inner automorphism of $\G$ to obtain an automorphism of $\G_\Lambda$.
This allows us to define the restriction map
$$\operatorname{Res} \colon \Out(\G;\G_\Lambda) \to \Out(\G_\Lambda)$$
where $\Out(\G;\G_\Lambda)$ denotes the subgroup consisting of all outer automorphisms that preserve $\G_\Lambda$ up to conjugacy. Well-definedness of this map follows from the structure of the normalizer of special subgroups of graph products.

It helps to make a sensible choice of $\Lambda$, so that the domain for $\operatorname{Res}$ can be taken to be $\Out^1(\G)$, $\Out^0(\G)$, or one of the other finite-index subgroups.
A common choice to achieve this is a maximal equivalence class, or the star of a maximal equivalence class, considering Lemma~\ref{lem:preservation of maximal stars}.

A \emph{factor map}, also known as a \emph{projection map}, is defined using the epimorphism $\kappa \colon \G \to \G_\Lambda$ which kills all vertices outside of $\Lambda$.
If the kernel $K$ of $\kappa$ (which is normally generated by $\G_{\Gamma\setminus\Lambda}$) is preserved by $\phi$, then we may define an automorphism of $\G_\Lambda$ by sending $\kappa(h)$, for any $h\in \G$, to $\kappa(\phi(h))$.
This allows us to define a factor map
$$\operatorname{Fact} \colon \Out(\G ; K) \to \Out(\G_\Lambda).$$

Again, we want to make the right choice of $\Lambda$ so that $\Out(\G;K)$ is as large as possible. For example, let $X$ be an equivalence class in $\Gamma$, and define:
$${\leq}X = \{v \in \Gamma \mid v\leq x \textrm{ for some $x\in X$} \}.$$
Then $\grp{\grp{\Gamma\setminus{\leq}X}}$ is preserved by all transvections and partial conjugations.
In particular, we may define:
$$\operatorname{Fact} \colon \Out^1(\G) \to \Out^1(\G_{\leq X}).$$
This can then be composed with the restriction map to $X$ (since $X$ is the unique maximal equivalence class in $\le X$), yielding
$$\operatorname{Res}\circ \operatorname{Fact} \colon \Out^1(\G) \to \Out^1(\G_X).$$
It is not hard to see this composition is in fact surjective.
This map was exploited in \cite{GuirardelSale-vastness}, and as there, allows us to deduce some properties concerning the quotients of finite index subgroups of $\Out(\G)$.

Recall that a group $G$ is said to have all \emph{finite groups involved} if for every finite group $H$ there is a finite index subgroup of $G$ that admits $H$ as a quotient.
It is \emph{SQ-universal} if for every countable group $C$ there is a quotient of $G$ that admits $C$ as a subgroup.
Also recall that $G$ is \emph{boundedly generated} if there exist $g_1,\ldots ,g_k \in G$ such that every element in $G$ can be written as $g_1^{n_1}\cdots g_k^{n_k}$ for some integers $n_1,\ldots , n_k$.

\begin{prop}
	If $(\Gamma,o)$ contains an equivalence class $X$ such that $\G_X$ is not abelian, then $\Out(\G)$ is SQ-universal, has all finite groups involved, and is not boundedly generated.
\end{prop}

\begin{proof}
	By using the map $\operatorname{Res}\circ \operatorname{Fact}$ above, we have a homomorphism onto $\Out^1(F_n)$, where $F_n = \G_X$ by Lemma~\ref{lem:eq classes infinite or finite}.
	Each of these properties hold for $\Out^1(F_n)$ and are inherited from quotients and finite index subgroups, so they pass to $\Out(\G)$.
	See \cite[Section 1.2]{GuirardelSale-vastness} for more details and references.
\end{proof}

\subsection{The standard representation}

The standard representation of $\Aut(\G)$ is obtained by acting on the abelianization $\bar{\G}$ of $\G$.
Denote it by: 
$$\rho \colon \Aut(\G) \to \Aut(\bar{\G}).$$
Note that $\rho$ factors through $\Out(\G)$.

The abelianization has the form $\bar{\G} \cong \Z^n\times T$, for some $n \geq 0$ and $T$ is a finite group.

\begin{lem}\label{lem:aut abelian}
	$\Aut(\bar{\G}) \cong T^n \rtimes (\GL(n,\Z) \times \Aut(T))$.
\end{lem}

\begin{proof}
	View $\bar{G}$ as a graph product of abelian groups in the form $\bar{G} = \G(\bar{\Gamma},o)$, where $\bar{\Gamma}$ is obtained from $\Gamma$ by adding all absent edges to the graph to make it complete.
	We may define a factor map $\operatorname{Fact}$ from $\Aut(\bar{\G})=\Out(\bar{\G})$ to $\GL(n,\Z)$ by killing all vertices of finite order (i.{}e.{} by killing $T$).
	We may also define a restriction map from $\Aut(\bar{\G})$ to $\Aut(T)$.
	Combining these, we get a homomorphism 
	$$\Aut(\bar{\G}) \to \GL(n,\Z) \times \Aut(T).$$
	The kernel is generated by the transvections $R_u^v$ were $o(u)=\infty$ and $o(v)<\infty$, which generate a subgroup isomorphic to $T^n$.
	Meanwhile, it is surjective, which we can see by looking at the image of the transvections not in the kernel. Further, since $\bar{\G}\cong \Z^n\times T$ there is a natural map $\GL(n,\Z) \times \Aut(T) \to \Aut(\bar{\G})$ which shows the short exact sequence we get is split.
\end{proof} 

The kernel of $\rho$ is denoted $\IA_\G$, and is sometimes referred to as the Torelli subgroup. 
In Section~\ref{sec:torelli} below we give a finite generating set for it when $\G$ is a graph product of finitely generated abelian groups.
When the vertex groups are finite, it is not so hard to find a generating set, and in fact it follows from results of Gutierrez--Piggott--Ruane \cite{GPR_automorphisms}.

\begin{prop}\label{prop:torelli gp finite}
	Suppose $\G$ is a graph product of finite abelian groups (i.e.{} $o(v) < \infty$ for all $v\in\Gamma$).
	Then $\IA_\G = \Aut^1_\infty(\G)$, and is generated by the set of partial conjugations of $\G$.
	
	Furthermore, the restriction of the standard representation to the subgroup generated by graph symmetries, transvections and factor automorphisms is injective.
\end{prop}

\begin{proof}
	We appeal to \cite[Theorem 3.1]{GPR_automorphisms}, stating that $\Aut(\G)$ is isomorphic to the semidirect product
	$\Aut^1_\infty(\G)~\rtimes~\Aut^2(\G)$,
	where $\Aut^2(\G)$ is the subgroup consisting of automorphisms which map each maximal special finite
	subgroup of $\G$ to a maximal special finite subgroup\footnote{In \cite{GPR_automorphisms}, they use $\Aut^1(W)$ to denote this group, and $\Aut^0(W)$ to denote what we refer to here as $\Aut^1_\infty(\G)$.}.
	
	We claim that $\rho$ restricted to $\Aut^2(\G)$ is an isomorphism.
	To see this, we use that any automorphism $\phi$ in $\Aut^2(\G)$ restricts to an isomorphism between maximal special finite subgroups. First note that if $K$ is a maximal special finite subgroup of $\G$, then $K$ is abelian and maps isomorphically to $\overline\G$. Thus, if $K, K'$ are maximal special finite subgroups of $\G$ and $\phi\in \IA_\G$ restricts to an isomorphism between $K$ and $K'$, we must have that $K=K'$. Hence $\phi$ restricts to the identity on $K=K'$. This implies $\rho$ is injective on $\Aut^2(\G)$.
	To see it is surjective, notice that $\Aut^2(\G)$ contains all graph symmetries, factor automorphisms, and transvections.
	As partial conjugations are trivial under $\rho$, surjectivity follows.
	
	Finally, the semidirect product structure of $\Aut(\G)$ above, together with this isomorphism, gives us that $\IA_\G = \Aut^1_\infty(\G)$.
\end{proof}

\section{Finite index subgroups}\label{sec:fi subgroups}

In the study of RAAGs, it is not hard to see that the transvections and partial conjugations generate a finite index subgroup of $\Aut(A_\Gamma)$.
For RACGs, the picture is transformed because the transvections can, up to finite index, be ignored: 
the partial conjugations alone generate a finite-index subgroup of $\Aut(W_\Gamma)$.
For more general graph products, there are a range of subgoups we can consider, which were defined in Section~\ref{sec:graph products}. We recall the definition and notation in the following proposition, the proof of which is the subject of this section.

\begin{prop}\label{prop:finite index subgroups}
	The following subgroups have finite index in $\Aut(\G)$.
	\begin{enumerate}
		\item $\Autfin(\G)$: the subgroup of all automorphisms that preserve maximal finite special subgroups of $\G$ up to conjugacy.
		\item $\Aut^1(\G)$: the subgroup generated by all partial conjugations and all transvections.
		\item $\Aut^1_\infty(\G)$: the subgroup generated by all partial conjugations and all transvections $R_u^v$ such that $u\leinf v$.
		\item $\Aut^0(\G)$: the subgroup generated by $\Aut^1(\G)$ and all factor automorphisms.
		\item $\Aut^0_\infty(\G)$: the subgroup generated by $\Aut^1_\infty(\G)$ and all factor automorphisms.
	\end{enumerate}
\end{prop}

We begin by proving the first part.

\begin{prop}
	The subgroup $\Autfin (\G)$ has finite index in $\Aut(\G)$.
\end{prop}

\begin{proof}
	Any element of $\Aut(\cal G)$ must permute the finite subgroups of $\cal G$. 
	However, every finite subgroup is conjugate into a finite special subgroup, and thus every maximal finite subgroup is conjugate to a maximal finite special subgroup. 
	Since there are only finitely many such special subgroups of $\cal G$, it follows that there are only finitely many such conjugacy classes. Now, $\Aut(\cal G)$ acts on the set of conjugacy classes of maximal finite subgroups, and $\Autfin(\cal G)$ is the kernel of this action. Thus, it is finite index in $\Aut(\cal G)$.
\end{proof}

In the next two subsections, we prove parts (4) and (3) (Corollaries~\ref{cor:Aut0fi} and~\ref{cor:Aut1fi} respectively) showing that 
$\Aut^0(\G)$ and $\Aut^1_\infty(\G)$ 
have finite index in $\Aut(\G)$. The two corollaries together prove the remaining parts of Proposition~\ref{prop:finite index subgroups}. The particular methods employed will be useful in what follows, so we prove these two corollaries in separate subsections.

\subsection{Detecting Graph Symmetries}
Duncan and Remeslennikov  showed that for a RAAG, $\Aut(A_\Gamma)$
is the semidirect product  $\Aut^0(A_\Gamma) \rtimes A$, where $A$ is a subgroup of the group of graph symmetries of the ``compressed graph''---obtained by fusing each equivalence class into one vertex \cite[Proposition 33]{DuncanRemeslennikovII}.
We also refer the reader to the work of Day and Wade \cite[Section 3.1]{DayWade}, who explicitly describe the quotient map in language and notation that is more consistent with that found here.

We apply this construction when $\Gamma$ has finite order vertices to detect when an automorphism can be written as products of generators not involving graph symmetries.
The steps in constructing the map are the same as for the RAAG case described by Day and Wade. We give an outline of the proof, focusing on issues that arise due to the introduction of finite order vertices. We refer the reader to \cite[Section 3.1]{DayWade} for details that we do not include.

We can partition the vertices of $\Gamma$ into equivalence classes for the partial order $\le$. 
First define a new graph $\Gamma_\le$ to have vertex set equal to the set of equivalence classes for $\le$, and define two vertices to be adjacent in $\Gamma_\le$ if and only if the corresponding equivalence classes in $\Gamma$ contain adjacent vertices.

We will define a map 
$$\Sigma \colon \Aut(\G) \to \Aut(\Gamma_\le).$$
This is done in a similar way to \cite{DayWade}, however care has to be taken with regards to the order of vertices.

Let $X$ be an equivalence class.
We consider the subgroups $\G_{\geq X}$ and $\G_{>X}$ generated respectively by those vertices $v$ such that $v\geq x$ for $x\in X$, excluding those $v\in X$ in the latter subgroup
(so that the quotient $\G_{\ge X} / \grp{\grp{\G_{>X}}}$ is isomorphic to $\G_X$).

\begin{lem}[c.{}f.{} {\cite[Proposition 31]{DuncanRemeslennikovII}}, {\cite[Proposition 3.1]{DayWade}}]
	\label{lem:autget}Let $\phi \in \Aut(\G)$. 
	There is an automorphism $\sigma$ of $\Gamma_\le$ such that for each equivalence class $X$  we have that
	\begin{itemize} 
		\item $\phi(\G_{\geq X})$ is conjugate to $\G_{\geq \sigma X}$,
		\item $\phi( \grp{\grp{\G_{>X}}} ) = \grp{\grp{\G_{>\sigma X}}}$.
	\end{itemize}	
\end{lem}

 We refer the reader to \cite[Proposition 3.1]{DayWade} for the proof in the RAAG case, which carries through to our situation. Essentially, one has to verify that the properties given in the lemma are preserved by taking a product of automorphisms. Thus it is enough to verify the properties hold for generators, which is straight-forward to do.

Recall from Lemma~\ref{lem:eq classes infinite or finite} that each equivalence class generates either a free group, a free abelian group, or a finite abelian $p$--group, for some prime $p$.
In Lemma~\ref{lem:autget}, if $X$ is free (resp.{} free abelian), then so is $\sigma X$, and if $X$ generates a finite abelian $p$--group, then so does $\sigma X$ (for the same $p$).

Define $\Sigma(\phi)(X) = (\sigma X)$, for $\varphi,\sigma$ as in Lemma~\ref{lem:autget}.

\begin{lem}
	The function $\Sigma(\phi)$ is an automorphism of the graph $\Gamma_\le$.
\end{lem}

\begin{prop}[c.{}f.{} {\cite[Proposition~3.3]{DayWade}}]\label{prop:getting graph symmetries}
	The map $\Sigma$ is a well-defined homomorphism with kernel $\Aut^0(\G)$.
\end{prop}

\begin{proof}
The proof is nearly identical to the RAAG case \cite[Proposition~3.3]{DayWade} however there is one key point where the orders of vertices come in.
Seeing that $\ker \Sigma $ contains $\Aut^0(\G)$ follows by looking at the generators of $\Aut^0(\G)$ and verifying they preserve the subgroup $\G_{\ge X}$ up to conjugacy, and the normal subgroup $\grp{\grp{\G_{>X}}}$, for each equivalence class $X$.
For the reverse inclusion, as in \cite[Proposition~3.3]{DayWade}, conjugating a factor automorphism, a transvection, or a partial conjugation by a labeled graph symmetry  gives a generator of the same type. 
Hence, given $\varphi \in \ker \Sigma$ as a word on the generating set of Theorem~\ref{thm:CGaut}, the graph symmetries can be shuffled to the end and we may write $\varphi = \sigma \varphi'$, with $\varphi'\in \Aut^0(\G)$ and $\sigma$ a graph symmetry. 
Then $1=\Sigma(\varphi) = \Sigma(\sigma)$. So $\sigma$ is an automorphism of the labeled graph that preserves the equivalence classes.
Given vertices $v,w$, it is noted in \cite[Proposition~3.3]{DayWade} that
$$\iota_v R_v^{w\m} \iota_v \iota_w R_w^v R_v^{w\m},$$
where $\iota_v$ is the inversion of the vertex $v$, is the automorphism swapping $v$ and $w$ and fixing all other vertices.
Thus to see $\sigma \in \Aut^0(\G)$, we need the elements in the above product.
However, since automorphisms are order preserving, $\sigma$ not just preserves the equivalence classes but also preserves orders of vertices.
Since for two equivalent vertices $v, w$,  both $R^v_w$ and $R^w_v$ exist (with $k=1$) if and only if $o(w)=o(v)$, we have the required transvections.
Meanwhile, the inversions are special cases of factor automorphisms, with the exception being when a vertex has order 2, in which case we replace the inversions with the identity.
\end{proof}

Since $\Aut(\Gamma_\le)$ is a finite group, this immediately leads to the following corollary.

\begin{cor}\label{cor:Aut0fi}
	The subgroup $\Aut^0(\G)\le \Aut(\G)$ has finite index.\end{cor}

\subsection{The Homomorphism D}\label{sec:homom D}
The homomorphism $D$ is an amalgamation of maps.
First fix an equivalence class $X$ in $\Gamma$ that contains vertices of infinite order. We will define a homomorphism $D_X$ as a composition of maps. The first map to take is the quotient map from $\Aut^0(\G)$ to $\Out^0(\G)$.
We can define the factor map $\Out^0(\G) \to \Out^0(\G_{\leq X})$, and compose this with the restriction to $X$, giving a homomorphism $\Out^0(\G) \to \Out^0(\G_X)$. Indeed, this is the map $\text{Res}\circ\text{Fact}$ defined in Section~\ref{subsec:ResFact}, and so it is onto.
From $\Out^0(\G_X)$, take the standard representation of this by acting on the abelianization $\bar{\G}_X$.
Using Lemma~\ref{lem:eq classes infinite or finite}, the image of this representation is isomorphic to $\GL(N_X,\Z)$, for some integer $N_X$.
From here we can use the determinant map onto $\Z_2$.
In summary, we have a homomorphism 
$$D_X \colon \Aut^0(\G) \to \Z_2.$$
We bundle all these maps together, along with the following map.
Let $\bar{G} \cong \Z^n \times  T$ be the abelianization of $\G$. 
Then the image of the standard representation is, up to isomorphism,  a subgroup of $T^n\rtimes (\GL(n,\Z) \times \Aut(T))$ by Lemma~\ref{lem:aut abelian}.
We define a map $D_0$ to be the composition of the standard representation, the inclusion into the semidirect product, and the quotient map onto $\Aut(T)$, giving
$$D_0 \colon \Aut^0(\G) \to \Aut(T).$$
Bundling all these maps together, and letting $\cal{E}$ denote the set of all equivalence classes in $\Gamma$ containing infinite order vertices, we get
\begin{equation}\label{eq:defn D}
D := D_0 \times \prod_{X\in\cal{E}} D_X \colon \Aut^0(\G) \to \Aut(T) \times \prod_{X\in\cal{E}}  \Z_2 .
\end{equation}
Since $\Gamma$ and $T$ are finite, the image of $D$ is a finite group.
The following therefore implies that $\Aut^1_\infty(\G)$ has finite index in $\Aut^0(\G)$, and hence in $\Aut(\G)$.

\begin{prop}\label{prop:kernel D}
	The kernel of $D$ is equal to $\Aut^1_\infty(\G)$.
\end{prop}

\begin{proof}
To see that $\Aut^1_\infty(\G) \subseteq \ker D$, we note that each generator of $\Aut^1_\infty(\G)$ lies in the kernel.
Indeed, since partial conjugations act trivially on the abelianization, they are all in $\ker D$.
Meanwhile, consider $u,v$ such that $u\leq_\infty v$.
If $o(v)<\infty$, then the image of $R_u^{v}$ under the standard representation lies in the  subgroup $T^n$ of $T^n\rtimes (\GL(n,\Z) \times \Aut(T))$. 
If $o(v) = \infty$, then it lies in $\GL(n,\Z)$.
In particular, in either case, it maps to the identity under $D_0$.
Now consider an equivalence class $X \in \cal{E}$. 
If at least one of $u$ or $v$ is not in $X$ the the transvection $R_u^{v}$ is killed by the composition of the factor and restriction maps to $\Aut^0(\G_X)$.
Otherwise, if $u,v\in X$, then $o(v)=\infty$, and $R_u^v$ is sent to an elementary matrix with determinant one under the standard representation of $\Aut^0(\G_X)$, and hence is killed by $D_X$.

We are left to show that $\ker D \subseteq  \Aut^1_\infty (\G)$.	
Let $\phi \in \ker D$.
Write $\phi$ as a product of generators for $\Aut^0(\G)$, namely $\phi = \gamma_1 \cdots \gamma_r$, where $\gamma_i$ is either a factor automorphism, a (power of a) transvection, or a partial conjugation.
As verified above, any generator in $\Aut^1_\infty(\G)$ is in $\ker D$, so we get $D(\phi) = D(\phi')$, where $\phi'$ is obtained from $\gamma_1\cdots \gamma_r$ by deleting partial conjugations and transvections in $\Aut^1_\infty(\G)$.
Hence it is enough to assume each $\gamma_i$ is a factor automorphism or a transvection $R_u^{v^k}$ with $o(u)<\infty$.

Let $f$ be a factor automorphism acting on the vertex $v$, and let $R_u^{v^k}$ be a transvection.
Since the order of $f(v^k)$ equals that of $v^k$, the automorphism $R_u^{f(v^k)}$ is well-defined,
and manual computation verifies the relation
\begin{equation*}
f R_u^{v^k} = R_u^{f(v^k)} f.
\end{equation*}
Thus we may assume there is $0\leq s\leq r$ such that $\gamma_1 , \ldots , \gamma_s$ are factor automorphisms of distinct vertices (if $s=0$ this list is empty), while $\gamma_{s+1},\ldots ,\gamma_r$ are transvections with finite order support (if $s=r$ this list is empty).
Suppose $\gamma_i$ and $\gamma_j$ are factor automorphisms of infinite order vertices $u$ and $v$, respectively, belonging to the same equivalence class.
As in the proof of \cite[Proposition 4.9]{WadeThesis} we may write
\begin{equation*}
\gamma_j = \gamma_i R_u^v R_v^{u^{-1}} \gamma_i R_u^{v^{-1}} \gamma_i R_u^v \gamma_i R_v^u R_u^{v^{-1}} \gamma_i.
\end{equation*}
Using this, we may replace each occurrence of $\gamma_j$ with $\gamma_i^5=\gamma_i$, after deleting the transvections, which are all in $\Aut^1_\infty(\G)$. 
Thus we may assume that the factor automorphisms $\gamma_1 , \ldots,\gamma_s$ act on at most one infinite order vertex from each equivalence class.
Then, similar to Wade's proof, by mapping onto the product of groups $\Z_2$, one copy for each equivalence class containing infinite order vertices,
we can realise that in order for $\phi$ to be in $\ker D$, none of the factor automorphisms $\gamma_1 , \ldots,\gamma_s$ can act on an infinite order vertex of $\Gamma$.

Now we focus on $D_0$, which comes about from compositions:
$$
\Aut^0(\G) \to \Aut^0(\bar{G}) \hookrightarrow T^n \rtimes (\GL(n,\Z) \times \Aut(T)) \to \Aut(T).$$
Consider the subgroup $H$ of $\Aut^0(G)$ consisting of all automorphisms that preserve the special subgroup $\G_f$ generated by the set of all finite-order vertices in $\Gamma$ and fix all infinite-order vertices.
Then $H$ can be viewed as a subgroup of $\Aut^0(\G_f)$. The abelianization of $\G_f$ is $T$, so the standard representation for $\Aut(\G_f)$ has image in $\Aut(T)$.
In particular, $D_0$ restricted to $H$ is the same map as the standard representation on $\Aut(\G_f)$ restricted to $H$.
By Proposition~\ref{prop:torelli gp finite}, the subgroup of $\Aut(\G_f)$ generated by graph symmetries, factor automorphisms, and transvections embeds into $\Aut(T)$.
In particular, as we have just shown, $\phi'$ is such an element. Thus, $\phi'$ can be in the kernel of $D_0$, and hence of $D$, only if $\phi'=1$. Thus, $\phi\in\Aut^1_\infty(\G)$, as desired.
\end{proof}

To conclude, we have the following.

\begin{cor}\label{cor:Aut1fi}The subgroup $\Aut^1_\infty(\G)$ is finite index in $\Aut(\G)$.\end{cor}

\section{The amalgamated projection}\label{sec:amalg proj}

Throughout this section, we assume that $\Gamma$ is connected and is not the star of a vertex. Equivalently, $\G$ is freely indecomposable and has trivial center.

\subsection{The amalgamated restriction and projection maps}

Given a $\le$--maximal vertex $v\in \Gamma$, we can define the restriction map
\[
R_v \colon \Out^1(\cal{G}) \to \Out(\cal{G}_{\st[v]}).
\]
By choice of $v$, this is a homomorphism.

We bundle these maps together to get the amalgamated restriction map
\[
R \colon \Out^1(\cal{G}) \to \prod_{[v] \textrm{ maximal}} \Out(\cal{G}_{\st[v]}).
\]

Since automorphisms will preserve $[v]$ up to conjugacy (as $v$ is maximal),
the restriction map $R_v$ defined above may be composed with a factor map to $\lk[v]$.
This gives homomorphisms
\[
P_v \colon \Out^1(\cal{G}) \to \Out(\cal{G}_{\lk[v]}).
\]
As with the restriction maps, we can combine these to a single map.
\[
P \colon \Out^1(\cal{G}) \to \prod_{[v] \textrm{ maximal}} \Out(\cal{G}_{\lk[v]}).
\]
We call $P$ the {amalgamated projection map}.
The aim of this section is to show this has finite-rank abelian kernel.

As per \cite{CV09}, we make the following definition.

\begin{defn}
	A \emph{leaf-like} transvection is a transvection $R_{u}^{v^k} \in \Out(\cal{G})$ where $v$ is a maximal vertex and $[v]$ is the only maximal equivalence class contained in $\lk(u)$.
	
	In this situation, we say $u$ is a \emph{leaf-like} vertex, and use the same adjective for the equivalence class containing $u$.
\end{defn}

We note that it follows from this definition that if $R_u^{v^k}$ is leaf-like, then $u$ and $v$ cannot be in the same $\leq$--equivalence class.

\begin{lem}\label{lem:leaf-like maximal}
	If $R_u^{v^k}$ is a leaf-like transvection then $[v]$ is an abelian equivalence class.
\end{lem}

\begin{proof}
	If $v'\in[v]$, then $v'\in \lk(u) \subseteq \st(v)$.
\end{proof}

Our next definition comes from \cite{CV11}, where it was called a ``$\hat{v}$--component.''

\begin{defn}
	An induced subgraph $C$ of $\Gamma$ is said to be a \emph{bridged $v$--component} if $C$ is a minimal union of connected components of $\Gamma \setminus \st(v)$ such that for every $u\in C$ if $\gamma$ is an edge path in $\Gamma$ starting at $u$, containing no edges in $\st(v)$ (but possibly including vertices from it), and ending at a vertex $w$, then $w$ is also in $C$.
\end{defn}

\subsection{The extended graph}
Define $\hat{\Gamma}$ as follows (compare with the relative cone graph of \cite[Section 3.2]{DayWade}).
First add two vertices $v_1,v_2$ whose links are equal to $\Gamma$.
For each maximal equivalence class $X$ in $\Gamma$, and each non-empty subset $S\subseteq \lk X$, add a vertex 
$X_S$, whose link is $X \cup S$.
We extend $o$ to $\hat{o}$ on $\hat{\Gamma}$ by defining 
$\hat{o}(v_i)$ and $\hat{o}(X_S)$, to be distinct primes that do not divide $o(v)$ for any $v\in \Gamma$.

We will denote $\G(\Gamma,o)$ by $\G$ and $\G(\hat{\Gamma},\hat{o})$ by $\Ghat$. We also denote stars by $\st_{\hat{\Gamma}}$ or $\st_\Gamma$, depending on which graph we are working in, and similarly for links.

We claim the restriction map
$$\hat{R} \colon \Out^{1}(\Ghat) \to \Out(\G)$$
is well-defined and has image equal to $\ker(P)$.

\begin{lem}
	The restriction map $\hat{R}$ is well-defined.
\end{lem}

\begin{proof}
	We first show that the restriction $\hat{R}$ is well-defined by checking that each generator of $\Out^1(\Ghat)$ preserves $\G$.
	
	First consider a transvection $R_u^{v^k}$. 
	If $u,v\in \Gamma$, or $u \notin \Gamma$, then $\cal{G}$ is preserved.
	The other case, with $u\in \Gamma$ and $v\notin \Gamma$, is not possible. 
	Indeed, for $u\in \Gamma$, we will have $v_1,v_2 \in \lk(u)$, but $v_i \notin \st(v)$ unless $v=v_i$.
	Thus all transvections preserve $\cal{G}$.
	
	Now consider a partial conjugation $\pi^v_C$.
	If $v\in \Gamma$, then $\cal{G}$ is preserved.
	So assume $v \notin \Gamma$.
	We claim either $C \cap \Gamma = \emptyset$, or $\Gamma \subseteq \st(v) \cup C$.
	Indeed, if $x\in C\cap\Gamma$, and $y\in \Gamma\setminus \st(v)$, then $x$ is connected to $y$ via $v_1$ and via $v_2$, at least one of which is outside $\st(v)$. So $y\in C$ also.
	Thus partial conjugations preserve $\cal{G}$.
\end{proof}

We now analyze how the partial order on $\Gamma$ differs to that on $\hat{\Gamma}$.
The main consequence is that only leaf-like transvections remain in the automorphism group for the extended graph. In particular, we also see that the equivalence classes in $\Gamma$ are dismantled to give equivalence classes in $\hat{\Gamma}$ that consist of only a single vertex.

For mathematical, if not aesthetic, clarity, we denote the partial orders $\leq$ and $\leinf$ on $\hat{\Gamma}$ by $\lehat$ and $\leinfhat$.

\begin{lem}\label{lem:preorder on extended graph}
Suppose $\Gamma$ is connected and is not the star of a vertex.
	Let $u,v$ be vertices in ${\Gamma}$.
	Then $u \lehatcomp  v$ if and only if there exists $k\in \N$ such that $R_u^{v^k}$ is leaf-like.
\end{lem}

\begin{proof}
	It is clear that a necessary condition for $u \lehatcomp v$ is $u\leq v$.
	Suppose $u\leq v$ but $R_u^{v^k}$ is not leaf-like. 
	This implies that either $u$ and $v$ are non-adjacent, $v$ is not a maximal vertex, $u$ and $v$ are equivalent, or $u$ is adjacent to a maximal equivalence class $X\not\ni v$.
	We will see that in each case, as is immediate in the last case, there is an equivalence class $X$ in $\lk_{\hat{\Gamma}}(u)$ that does not contain $v$.
	Then the vertex $X_{\{u\}}$ witnesses $u \nlehatcomp v$.
	
	Firstly, if $v$ is not adjacent to $u$, then by Lemma~\ref{lem:infty-stars cover Gamma} (2) there is a maximal equivalence class $X$ in $\lk_\Gamma(u)\cap\lk_\Gamma(v)$.
	Next, if $v$ is adjacent to $u$, but it is not maximal, then any maximal equivalence class $X$ dominating $v$ must also be in the star of $u$. 
	If $v$ and $u$ are adjacent, and both maximal, then they must be in the same equivalence class, since $v$ dominates $u$.
	By Lemma~\ref{lem:infty-stars cover Gamma} (1), there is a maximal equivalence class $X$ in $\lk_\Gamma([u])$.
	
	On the other hand, suppose $R_u^{v^k}$ is leaf-like.
	We want to check that any new vertex in $\hat{\Gamma}$ that is adjacent to $u$ is also adjacent to $v$.
	This is certainly true for $v_1$ and $v_2$.
	So suppose that $X_S$ is a new vertex and $u \in \lk_{\hat{\Gamma}}(X_S)$.
	By the definition of leaf-like, we cannot have $u\in X$.
	Thus, $X$ is contained in $\lk_\Gamma(u)$, and since $R_u^{v^k}$ is leaf-like, we must have $v\in X$.
	Hence $v$ is also in $\lk_{\hat{\Gamma}}(X_S)$.
\end{proof}

\begin{rem}\label{rem:equivclass breakup} Suppose that $x_1,x_2\in \Gamma$ with $x_1,x_2$ in the same $\le$--equivalence class, $X$. Then there exists a maximal $\le$--equivalence class, $Y$ in $\lk_\Gamma(X)$ by Lemma~\ref{lem:infty-stars cover Gamma}. Thus, in $\hat\Gamma$, $Y_{\{x_1\}}\in\lk_{\hat\Gamma}(x_1)\setminus\lk_{\hat\Gamma}(x_2)$ and $Y_{\{x_2\}}\in\lk_{\hat\Gamma}(x_2)\setminus\lk_{\hat\Gamma}(x_1)$. Thus, $x_1$ and $x_2$ are not equivalent under $\lehat$.\end{rem}

We now complete our look at $\hat\leq$ by considering vertices not in $\Gamma$.

\begin{lem}\label{lem:nodomination}
	Suppose $\Gamma$ is connected and is not the star of a vertex.
	Let $v\not\in\Gamma$. Then $v$ does not dominate any vertex in $\Gamma$.
 That is, there is no vertex $u\in \Gamma$ such that $u\ \lehat\ v$.
\end{lem}
\begin{proof}
	First suppose that $v\in\{v_1, v_2\}$. For any $u\in \Gamma$, $v_1, v_2\in\st(u)$, but $v_i\not\in\st(v_j)$ for $i\neq j$. Hence, $v$ does not dominate $u$.
	Now suppose that $v\not\in\{v_1, v_2\}\cup \Gamma$. Let $u\in\Gamma$ be such that $u\ \lehat\ v$. Then $\lk_{\hat\Gamma}(u)\subseteq\st_{\hat\Gamma}(v)$. Let $X$ be a $\le$--maximal equivalence class in $\lk_\Gamma(u)$, then $X$ is in the star of $v$, and further $v=Y_S$ for some maximal equivalence class $Y$ in $\Gamma$ and $\{u\}\cup X\subseteq Y\cup S= \lk_\Gamma(v)$.
	However, $X_{\{u\}}\in \lk_{\hat{\Gamma}} (u)$, and thus since $X_{\{u\}}$ is not adjacent to any vertex in $\hat\Gamma\setminus\Gamma$, we must have that $X=Y$ and $S=\{u\}$.
	Further, if $\{u\}\subsetneq S'\subseteq \lk_\Gamma(Y)$, then $X_{\{u\}}$ and $X_{S'}$ are non-adjacent. But $u\in S'$ and so $X_{S'}\in\lk_{\hat\Gamma}(u)\subseteq\st_{\hat\Gamma}(v)$. Thus, $\lk_\Gamma(X)=\{u\}$ and $\Gamma = X\star \{u\}$. This is a contradiction, since $\Gamma$ is connected and not the star of a single vertex.
\end{proof}
	
\begin{lem}\label{lem:hat eq singletons}
	Suppose $\Gamma$ is connected and is not the star of a vertex.
	Every $\lehat$--equivalence class in $\hat{\Gamma}$ contains a single vertex.
\end{lem}

\begin{proof}
	Let $X$ be an equivalence class in $\hat{\Gamma}$.
	By Remark~\ref{rem:equivclass breakup}, $X$ can contain at most one element from $\Gamma$.
	By Lemma~\ref{lem:nodomination}, if $X$ contains a vertex from $\Gamma$ then it cannot contain vertex not in $\Gamma$. 
	So if $X\cap \Gamma \neq \emptyset$, then $X$ is a singleton.
	
	On the other hand, if $X$ consists only of new vertices, then by the choice of $\hat{o}(v)$ as distinct primes for vertices $v\in\hat{\Gamma}\setminus \Gamma$, $X$ must once again be a singleton.
\end{proof}

We now study partial conjugations in $\Out^1(\Ghat)$.
The following result shows that all such partial conjugations restrict to $\G$ to give partial conjugations with a bridged component as its support.

\begin{lem}\label{lem:PC in Rhat image}
	Suppose $\Gamma$ is connected and is not the star of a vertex.
	Let $v \in \hat{\Gamma}$ and ${C}$ be a subset of ${\Gamma}$.
	Then $C$ is a bridged $v$--component in $\Gamma$ if and only if $C=\Gamma\cap\hat{C}$ for some connected component $\hat{C}$ of $\hat{\Gamma}\setminus\st_{\hat{\Gamma}}(v)$.
\end{lem}

\begin{proof}
	We first show that $C=\Gamma\cap\hat{C}$ must be a bridged $v$--component. 
	Suppose that $a,b \in C$.
	There must be a path in $\hat{\Gamma}$ from $a$ to $b$ which avoids $\st_{\hat{\Gamma}}(v)$. 
	Let $p$ be such a path,
	and let $a=p_0, p_1, \ldots, p_k=b$ be the sequence of vertices along this path. 
	If $p_i\in \Gamma$ for all $i$,  then $a,b$ are in the same component of $\Gamma\setminus \st_\Gamma(v)$, and are thus in the same bridged $v$--component. 
	So, without loss of generality, let $j$ be an index so that $p_j\not\in \Gamma$. 
	Then $p_j=X_S$ for some maximal equivalence class $X$ of $\Gamma$ and some subset $S\subseteq \lk_\Gamma X$, and further $p_{j-1}, p_{j+1}\in X \cup S$.
	If both are in $X$, we must have that $p_{j-1}$ and $p_{j+1}$ are in the same component of $\Gamma\setminus \st_\Gamma(v)$ by Lemma~\ref{lem:maximal eq classes not star separated}.
	Thus, we can replace this subpath with a path in $\Gamma$ that avoids $\st_\Gamma(v)$. 
	If one is in $X$ and the other in $S$, then $p_{j-1}, p_{j+1}$ are actually adjacent in $\Gamma$,  
	so we can remove the vertex $p_j$ creating a subpath that remains in $\Gamma$. 
	If both are in $S$, we can replace $p_j$ with a vertex $x\in X$, creating a subpath that stays in $\Gamma$, but may enter $\st_\Gamma(v)$, if $X\subseteq \st_\Gamma(v)$. 
	However, $p_{j-1}, p_{j+1}\not\in\st_\Gamma(v)$, so the new subpath contains no edge in $\st_\Gamma(v)$. 
	Thus, we can replace $p$ with a path in $\Gamma$ that connects $a$ and $b$ and contains no edge in $\st_\Gamma(v)$, implying $a$ and $b$ are in the same bridged $v$--component.
	
	Now we show that every bridged $v$--component $C$ in $\Gamma$ arises this way. 
	Define $\hat{C}$ to be the union of all connected components of $\hat{\Gamma} \setminus \st_{\hat{\Gamma}}(v)$ that intersect $C$.
	The claim is that only one connected component is required.
	Suppose $a,b \in C$. There is a path from $a$ to $b$ in $\Gamma$ that contains no edges in $\st_\Gamma(v)$.
	Choose such a path $p$ of minimal length, label the vertices of $p$ by $a=p_0, p_1, \ldots, p_k=b$.
	If it contains no vertices in $\st_\Gamma(v)$ then $a,b$ are in the same connected component of $\hat{\Gamma} \setminus \st_{\hat{\Gamma}}(v)$.
	So we may assume that there is at least one vertex $p_i$ in $\st_\Gamma(v)$. 
	However, we know that both vertices adjacent to $p_i$ in $p$ are not in $\st_\Gamma(v)$, and we can assume that they are not adjacent.
	By Lemma~\ref{lem:infty-stars cover Gamma}~(2),
	we may assume that $p_i$ is maximal in $\Gamma$.
	Then we can construct a new path from $a$ to $b$ in $\hat{\Gamma}$ that avoids $\st_{\hat{\Gamma}}(v)$ by replacing any such $p_i$ with $Y_{\{p_{i-1},p_{i+1}\}}$, where $Y$ is the equivalence class of $p_i$.
\end{proof}

\begin{lem}\label{lem:Rhat into kerP}
	Suppose $\Gamma$ is connected and is not the star of a vertex.
	The image of $\hat{R}$ is contained in the kernel of $P$.
\end{lem}

\begin{proof}
	The domain $\Out^1(\Ghat)$ is generated by transvections and partial conjugations.
	First check the transvections.
	By Lemma~\ref{lem:preorder on extended graph}, 	the only transvection that could appear in the image are leaf-like transvections.
	It is immediate from the definition of leaf-like that any leaf-like transvection is in the kernel of $P$.
	
	Now consider a partial conjugation $\pi^v_C$ in the domain of $\hat{R}$. If $v\notin \Gamma$, then $\hat{R}(\pi^v_C)$ is trivial. So we may assume $v\in \Gamma$.
	We want to show that for each maximal equivalence class $X$ in $\Gamma$ either $\st_\Gamma X \subseteq C \cup \st_\Gamma(v)$, or $\st_\Gamma X \cap C = \emptyset$.
	
	Suppose that $u \in \st_\Gamma X \cap C$, and let $w$ be any other vertex of $\st_\Gamma X$.
	First assume that $v \in X$. Since $u \in C$ we must have that $X$ is free and $u\in X$. 
	Then either $w\in \lk_\Gamma X$ and hence in $\st_\Gamma(v)$, or $w\in X$.
	As $\Gamma$ is connected, there must be a vertex $y$ in $\lk_\Gamma X$, and then the path $u,y,w$ has no edge in $\st_\Gamma(v)$, so $u,w$ are in the same bridged $v$--component. By Lemma~\ref{lem:PC in Rhat image} this bridged $v$--component is $C\cap\Gamma$. 
	
	Now assume $v \notin X$.
	If either $u$ or $w$ is not in $X$, then in $\hat{\Gamma}$ we have a two-edge path $u,X_S,w$,
	where $S$ consists of $u$ and $w$, minus either if it is in $X$.
	This path is outside of $\st_{\hat{\Gamma}}(v)$, so $w\in C$.
	
	Consider the case where $u$ and $w$ are both in $X$. If $X$ is an abelian equivalence class, then $w\in C\cup\st_\Gamma(v)$. Thus, suppose that $u$ and $w$ are non-adjacent. 
	As $\Gamma$ is connected, $\lk_\Gamma X$ is nonempty.
	If $v\in \lk_\Gamma X$ then $w\in \st_\Gamma(v)$ and we are done.
	So we may assume in particular that there is some vertex $y$ in $\lk_\Gamma X$ that is distinct from $v$.
	Then we have a two-edge path $u,X_{\{y\}},w$ whose edges and middle vertex lie outside $\st_{\hat{\Gamma}}(v)$, implying $w \in C\cup\st_\Gamma(v)$.
\end{proof}

It remains to show that the kernel of $P$ is contained in the image of $\hat{R}$.
Given an outer automorphism in $\ker P$, we construct an element of its pre-image under $\hat{R}$. In doing this,we need to prove that our constructed automorphism is in $\Aut^1(\Ghat)$.
For this, we use Proposition~\ref{prop:getting graph symmetries}, paired with the following.

\begin{lem}\label{lem:equivalence classes preserved by ker P}
	Suppose $\Gamma$ is connected and is not the star of a vertex.
	Let $x$ be a vertex in $\Gamma$, and $\Phi \in \ker P$. 
	Either
	\begin{enumerate}
		\item if $x$ is not leaf-like, then $\Phi$ preserves $\G_{\{x\}}$;
		\item if $x$ is leaf-like, with dominating adjacent maximal equivalence class $Y$, then $\Phi$ preserves $\G_{\{x\}\cup Y}$.
	\end{enumerate}
\end{lem}

\begin{proof}
	By Lemma~\ref{lem:infty-stars cover Gamma} (1), there is at least one maximal equivalence class $Y$ in $\lk_\Gamma (x)$.
	By Lemma~\ref{lem:preservation of maximal stars}, there is a representative $\phi$ of $\Phi$ such that $x$ is mapped to a word on $\st_\Gamma Y$, which we can write as $yw$, with $y \in \grp{Y}$ and $w \in \grp{\lk_\Gamma Y}$, since $Y$ commutes with $\lk_\Gamma Y$.
	Since $\Phi \in \ker P$, we must have $w = x$.
	
	If there are two distinct maximal equivalence classes in $\lk_\Gamma(x)$, then the above shows that each $x\in X$ is fixed by $\phi$. In particular, $\G_{\{x\}}$ is preserved.
	
	Otherwise $x$ is leaf-like, with $Y$ a maximal equivalence class.
	Then $\phi(x) = yx \in \G_{\{x\}\cup Y}$.
	Meanwhile, Lemma~\ref{lem:preservation of maximal stars} also tells us that $\G_Y$ is preserved by $\phi$, and so $\G_{\{x\}\cup Y}$ is preserved by $\phi$.
\end{proof}

\begin{lem}\label{lem:ker P in image Rhat}
Suppose $\Gamma$ is connected and is not the star of a vertex.
	The kernel of $P$ is contained in the image of $\hat{R}$.
\end{lem}

\begin{proof}
	Let $\Phi$ be in the kernel of $P$. Take $\phi$ to be a representative of $\Phi$.
	By Lemma~\ref{lem:preservation of maximal stars}, for each maximal equivalence class $X$, there exists $g_X\in \G$ such that $\phi(\grp{X}) = g_X\grp{X}g_X^{-1}$ and $\phi(\grp{\st X}) = g_X \grp{\st X} g_X^{-1}$.
	Define $\hat{\phi}$ on $\Ghat$, by acting on the vertices of $\hat{\Gamma}$ as follows:
	$$\hat{\phi}(v)=
	\begin{cases} 
	\phi(v)     & \text{ if $v$ is a vertex of $\Gamma$},\\
	g_X v g_X\m & \text{ if } v=X_S,\text{ for a maximal equivalence class $X$, and $S\subseteq \lk_\Gamma X$,} \\
	v           & \text{ if } v\in\{v_1, v_2\}.\end{cases}$$
	We first show that $\hat{\phi}$ is a homomorphism by verifying each defining relation is preserved.
	Indeed, each order relation on generators is preserved by the definition of $\hat{\phi}$.
	Also, all commutator relations between vertices in $\Gamma$ are preserved, and it is immediate that the image of $v_i$ commutes with the image of any vertex of $\Gamma$, for $i=1,2$.
	The only relation left to check is $[X_S , u] = 1$ for $u \in X\cup S$, for any maximal equivalence class $X$, and $S \subseteq \lk_\Gamma X$. 
	First assume $u\in S$.
	Since $\grp{X}$ commutes with $\grp{\lk_\Gamma X}$, we can write $\phi(u) = g_Xywg_X\m$, for $y\in\grp{X}$, and $w\in\grp{\lk_\Gamma X}$.
	Since $\Phi$ is in the kernel of $P$, we have $w=u$.
	This gives 
	$$[\phi(X_S) , \phi(u)] = [g_X X_S g_X\m , g_X y u g_X\m] = g_X[X_S , yu]g_X\m = 1$$
	since $yu$ commutes with $X_S$ as $X\cup \{u\}$ is in the link of $X_S$.
	Now assume $u\in X$.
Then $\phi(u) = g_X y g_X^{-1}$ for some $y\in \grp{X}$.
But $X \subset \lk_{\hat{\Gamma}}(X_S)$, so $\hat{\phi}(X_S) = g_X X_S g_X^{-1}$ commutes with $g_X y g_X^{-1}$ as required.	
	
	To verify that $\hat{\phi}$ is an automorphism, we note that its inverse can be constructed in a similar way using $\phi\m$ instead of $\phi$, and $g_X\m$ in place of $g_X$.
	
	We will use $\hat{\phi}$ to show that there is an automorphism $\phi_1$ in $\Aut^1(\Ghat)$ with $\hat{R}(\phi_1) = \phi$ (we're abusing notation here of course by applying $\hat{R}$ to an automorphism).
	We first apply Proposition~\ref{prop:getting graph symmetries} to see  that $\hat{\phi}$ is in $\Aut^0(\Ghat)$.
	Let $X$ be an equivalence class in $\hat{\Gamma}$. 
	Then by Lemma~\ref{lem:hat eq singletons}, $X = \{x\}$ for some $x\in\hat{\Gamma}$.
	We claim $\hat{\phi}$ preserves $\Ghat_{\hat{\geq}X}$ up to conjugacy, and so, by Proposition~\ref{prop:getting graph symmetries}, $\hat{\phi}$ is in $\Aut^0(\Ghat)$.
	
	To prove the claim, first suppose $x\in \Gamma$. By Lemmas~\ref{lem:preorder on extended graph} and~\ref{lem:nodomination}, if $x$ is not leaf-like, then the set $\hat{\geq}X$ is equal to $X$.
	If on the other hand $x\in \Gamma$ is leaf-like, then there is a unique maximal equivalence class $Y$ in $\st_\Gamma(x)$. 
	Again by Lemmas~\ref{lem:preorder on extended graph} and~\ref{lem:nodomination}, $\hat\geq X =X\cup Y$.
	In either case, Lemma~\ref{lem:equivalence classes preserved by ker P} implies $\Ghat_{\hat{\geq}X}$ is preserved, up to conjugacy, by $\hat{\phi}$.
	
	Now suppose $x\notin \Gamma$. Using again the definition of $\hat{o}$ on $\hat{\Gamma}\setminus \Gamma$, no vertex in $\hat{\Gamma}$ can dominate $x$. Further, again by choice of $o(x)$, there is no vertex in $\Gamma$ that dominates $v$. Thus, $\hat{\geq}X=X$ and it follows then from the definition of $\hat{\phi}$ that $\Ghat_{\hat{\geq}X}$ is also preserved, up to conjugacy.
	
	Manual computation verifies that the conjugate of any partial conjugation or transvection by a factor automorphism is again (a power of) a partial conjugation or transvection, respectively.
	Thus we can write $\hat{\phi}$ as $f \phi_1$, for $\phi_1 \in \Aut^1(\Ghat)$, and $f$ a product of factor automorphisms.	
	Since $\hat{R}(f) =\hat{R}(\hat{\phi}\phi_1\m) \in \ker P$ by Lemma~\ref{lem:Rhat into kerP},
	we deduce that $f$ acts trivially on $\Gamma$ (using Lemma~\ref{lem:infty-stars cover Gamma}~(1), that every vertex is in the link of some maximal equivalence class), so $\hat{R}(f)=1$.
	Thus $\hat{R}(\phi_1) = \hat{R}(f\m \hat{\phi}) = \phi$ and, up to replacing $\hat{\phi}$ with $\phi_1$, we obtain  an automorphism in $\Aut^1(\hat{G})$ whose image is $\phi$.	
(We note that we have abused notation: technically we apply $\hat{R}$ to the outer automorphisms containing $f,\hat{\varphi}\varphi_1\m, \varphi_1, f\m \hat{\varphi}$ above).
\end{proof}

We have now proved that $\Im \hat{R} = \ker P$. To conclude this section, we use the structure of $\Ghat$ and its automorphism group to give us a precise description of the kernel of $P$.

\begin{thm}\label{thm:kerP_abelian}
	Suppose $\Gamma$ is connected and is not the star of a vertex.
	The kernel of $P$ is abelian, generated by the set of leaf-like transvections and partial conjugations $\pi^v_C$ where $C$ is a bridged $v$--component.
\end{thm}

\begin{proof}
	Lemmas~\ref{lem:Rhat into kerP} and \ref{lem:ker P in image Rhat} imply that the image of $\hat{R}$ is equal to $\ker P$.
	By Lemmas~\ref{lem:preorder on extended graph} and \ref{lem:PC in Rhat image}, the image of $\hat{R}$ is generated by leaf-like transvections and partial conjugations of the form $\pi^v_C$ where $C$ is a bridged $v$--component.
	
	To see that it is abelian, we first claim that there is no SIL $(u,v\mids w)$ in $\hat{\Gamma}$ with $u,v,w\in\Gamma$, so all the partial conjugations in the image of $\hat R$ commute by Lemma~\ref{lem:partial conj not commute}.
	Indeed, suppose that $(u,v\mids w)$ were such a SIL.
	Firstly, observe that $(u,v\mids w)$ also forms a SIL in $\Gamma$.
	As $\Gamma$ is connected we may assume that $w$ is adjacent to $\lk_\Gamma(u)\cap\lk_\Gamma(v)$ in $\Gamma$, and in particular that $\lk_\Gamma(u) \cap \lk_\Gamma(w)$ is non-empty.
	By Lemma~\ref{lem:infty-stars cover Gamma} (2), $\lk_\Gamma(u) \cap \lk_\Gamma(w)$ contains a maximal $\le$--equivalence class $X$. Since $u$ and $v$ are non-adjacent, by the definition of a SIL, we have that $v\not\in X$.
	Thus in $\hat{\Gamma}$ we get a two-edge path $u,X_{\{u,w\}},w$ from $u$ to $w$ avoiding $\lk_{\hat{\Gamma}}(u)\cap\lk_{\hat{\Gamma}}(v)$, a contradiction.

	Now consider two leaf-like transvections $R_u^{v^k}$ and $R_x^{y^j}$. 
	These will commute if $u,v,x,y$ are distinct.
	Suppose $u=x$. 
	Then $v$ and $y$ are in the same equivalence class, by the definition of being leaf-like.
	But then $v$ and $y$ commute, and hence the transvections commute, by Lemma~\ref{lem:leaf-like maximal}.
	Now suppose $x=v$. Since $v$ is maximal, we must have that $v$ and $y$ are equivalent. Then $[y]$ is not contained in $\lk_\Gamma(v)$, so $R_v^{y^j}$ is not leaf-like.
	If $v=y$, it is easy to see the transvections commute. 
	Hence all leaf-like transvections pairwise commute.
	
	Now we take a leaf-like transvection $R_u^{v^k}$ and a partial conjugation $\pi^x_C$ with $C$ a bridged $x$--component.
	Since $u$ and $v$ are adjacent, up to multiplying by an inner automorphism and replacing $\pi^x_C$ with its inverse, we may assume $u,v\notin C$.
	If $x\neq u$, then they commute. So we may assume $x=u$.
	We claim there is only one bridged $x$--component of $\Gamma$, so $\pi^x_C$ is inner.
	Suppose there are two. Then there are vertices $a,b$ such that any path between them involves an edge from $\st_\Gamma(x)$.
	Take such a path of minimal length, and denote its vertices by $a=p_0,p_1,\ldots ,p_r=b$.
	Let $p_i$ be the first vertex in $\st(x)$.
	By Lemma~\ref{lem:infty-stars cover Gamma} (2), there is a maximal equivalence class in $\lk_\Gamma(x)\cap\lk_\Gamma(p_{i-1})$.
	Since $R_x^{v^k}$ is leaf-like, this implies $v\in \lk_\Gamma(x)\cap \lk_\Gamma(p_{i-1})$.
	Similarly, let $p_j$ be the last vertex in the path that is inside $\st_\Gamma(x)$. Then we also have $v\in \lk_\Gamma(x)\cap\lk_\Gamma(p_{j+1})$.
	By minimality of the length of the path, we must have $i=j$, else we could shorten the path by going from $p_{i-1}$ to $v$ to $p_{j+1}$.
	This implies that $a,b$ are actually in the same bridged $x$--component, a contradiction. So $\pi^x_C$ is trivial in $\Out(\G)$; the result follows.
\end{proof}

\section{Applications of the Amalgamated Projection}\label{sec:TARF}
We now turn to the proofs of Theorems~\ref{thms:OutGPTA} and \ref{thms:OutGPRF}, of the Tits Alternative and residual finiteness.

\subsection{The Tits Alternative}

For the outer automorphism group of a RAAG, a substantial obstacle in proving the Tits Alternative holds  was the case when the defining graph is disconnected, i.{}e.{} when $A_\Gamma$ decomposes as a free product.
Horbez dealt with this situation in full generality \cite{HorbezTA}, which we now briefly describe.

Recall that given a finitely generated group $G$, its \emph{Grushko decomposition} is a splitting of $G$ as a free product
$G = G_1\ast \ldots \ast G_n\ast F$
such that each $G_i$ is non-trivial, not isomorphic to $\Z$, and freely indecomposable, and $F$ is a finite rank free group.

\begin{thm}[Horbez \cite{HorbezTA}]\label{thm:HorbezFP}
	Let $G_1\ast \ldots \ast G_n\ast F$ be the Grushko decomposition of a group $G$.
	
	If each of the groups $G_i$ and $\Out(G_i)$ satisfy the Tits Alternative, then $\Out(G)$ does too.
\end{thm}

Since satisfying the Tits Alternative is stable under graph products \cite{AntolinMinasyanTA}, it is enough for us to consider the case when $\Gamma$ is a connected graph.

Throughout Section~\ref{sec:amalg proj} we made the assumption that $\G$ has trivial center. This was so we could apply Lemma~\ref{lem:infty-stars cover Gamma}.
The proof of Theorem~\ref{thms:OutGPTA}
has two stages. We first deal with the special case when $\Gamma$ is the star of a vertex, and later apply the machinery of Section~\ref{sec:amalg proj} to tackle the case of full generality.

\subsubsection{When $\Gamma$ is the star of a vertex}
This is the special case in which we can write $\Gamma$ as a join $\Gamma'\ast K$, where $K$ is a clique and $\Gamma'$ is not the star of any vertex. Then $\G=\G_{\Gamma'}\times \G_K$.

Every element of $\Out^1(\G)$ preserves $\G_K$, so we may define a restriction map 
$$R_K\colon \Out^1(\G)\to \Out(\G_K)$$ 
and a projection map $P_K\colon \Out^1(\G)\to \Out(\G_{\Gamma'})$, induced by taking the quotient of $\G$ by $\G_K$. 
We combine these and consider the map
$$E_K = P_K \times R_K \colon \Out^1(\G)\to \Out(\G_{\Gamma'})\times\Out(\G_K).$$

\begin{prop}\label{prop:ker_star} 
	The kernel of $E_K$ is an abelian group generated by transvections $R^{v^k}_u$, where $v\in K$ and $u\in\Gamma'$. 
\end{prop}

\begin{proof} Let $\{w_1, \ldots, w_n\}$ be the vertex set of $K$. Consider $\Phi\in \ker(E_K)$. Then $\Phi\in\ker(R_K)$ and, since $\G_K$ is the center of $\G$, for any representative $\phi$ of $\Phi$, $\phi(w_i)=w_i$ for all $w_i\in K$.
	Furthermore, for each such $\phi$ and any $u\in \Gamma'$, we have that $\phi(u)=g_u h_u$, where $g_u\in\gspan{\Gamma'}$ and $h_u\in\gspan{K}$. Since $\Phi\in\ker(P_K)$, we can choose a representative so that $g_u=u$ for all $u\in \Gamma'$. Thus, $\phi(u)=uw_1^{r_1}\cdots w_n^{r_n}$, with $r_i=r_i(u)$ all integers.
	
	Note that $o(u)=o(\phi(u))$, so if $o(u)=p^i$ for some prime $p$, then we must have that $o(w_1^{r_1}\cdots w_n^{r_n})\mid p^i$. Thus, whenever $o(w_j)=\infty$ or $p\nmid o(w_j)$, we get $r_j=0$. Further, if $o(w_j)=p^\ell$ with $\ell>i$, we must have that $p^{\ell-i}\mid r_j$. In this case, set $k_j=p^{\ell-i}$ and $t_j=r_jp^{i-\ell}$, otherwise set $k_j=1$ and $t_j=r_j$.
	
	Thus, we can see 
	$$\phi=\displaystyle\prod_{u\in \Gamma'}\left(\prod_{j=1}^n\left(R^{w_j^{k_j}}_u\right)^{t_j}\right).$$
	The choice of each integer $k_j$ ensures each transvection in the product is an automorphism.
	The order of the transvections is not important since they all commute, implying also that the kernel is abelian.
\end{proof}

As a consequence of Proposition~\ref{prop:ker_star}, the kernel of $E_K$ is isomorphic to a subgroup of $\G_K^{\abs{\Gamma'}}$.
In contrast to the RAAG case, see \cite[Proposition 4.4]{CV09}, in general it may not be the whole group since not every transvection is necessarily permissible --- order of elements must be considered. 

\begin{prop}\label{prop:star_SES}
	There  is an abelian subgroup $\Tr$ of $\Out^1(\G)$ such that
	$$\Out^1(\G)\cong \Tr\rtimes (\Out^1(\G_{\Gamma'})\times\Out^1(\G_K)).$$
\end{prop}

\begin{proof} 
	The semidirect product comes from the short exact sequence with quotient map $E_K$.
	The subgroup $\Tr$ is the kernel of $E_K$, which is abelian by Proposition~\ref{prop:ker_star}.
	We now need to show that the image of $E_K$ is precisely $\Out^1(\G)\times\Out^1(\G_K)$.
	
	First consider a transvection  $R^{v^k}_u$. 
	If $v\in K$ and $u\in \Gamma'$, then $E_K(R^{v^k}_u)=(1,1)$. 
	If $u,v\in K$, then $E_K(R^{v^k}_u)=(1, R^{v^k}_u)$.
	If $v\in \Gamma'$, then $u\in \Gamma'$ as well, and $E_K(R^{v^k}_u)=(R^{v^k}_u, 1)$.
	
	Now consider a (non-inner) partial conjugation $\pi^u_C$. Then $u\in \Gamma'$, and so $E_K(\pi^u_C)=(\pi^u_C, 1)$.
	
	The image of $E_K$ is generated by these elements, making is clear that the image is precisely $\Out^1(\G_{\Gamma'})\times\Out^1(\G_K)$.
	
	That the sequence is split is clear, since $\G\cong \G_{\Gamma'}\times\G_K$, and so any pair of automorphisms $\phi_1\in\Aut^1(\G_{\Gamma'})$, $\phi_2\in \Aut^1(\G_K)\cong\Out^1(\G_K)$ lifts to an automorphism of $\G$.
	Further, such an automorphism is inner if and only if $\phi_1$ is inner and $\phi_2$ is the identity (since $\G_K$ is abelian and central). 
	Thus, the sequence is split, and so $\Out^1(\G)$ decomposes as a semi-direct product.
\end{proof}

\subsubsection{The general case}
We are now ready to prove Theorem~\ref{thms:OutGPTA}. As in \cite{CV11,HorbezTA}, the proof is an induction on the size of $\Gamma$, using the set-up above.

\begin{proof}[Proof of Theorem~\ref{thms:OutGPTA}] First note that satsifying the Tits' alternative is stable under subgroups, finite index supergroups, direct products, and extensions by abelian groups. In particular, we can restrict our attention to $\Out^1(\G_\Gamma)$. 
	
	We proceed by induction on the number of vertices in $\Gamma$.
	If $\Gamma$ consists of a single vertex, then the Tits Alternative holds.
	
If $\Gamma$ has more than one vertex but is not connected, then $\G$ splits as a free product with each free factor and its outer automorphism group satisfying the Tits alternative by \cite{AntolinMinasyanTA} and the inductive hypothesis respectively.
Then by Horbez' theorem for free products, Theorem~\ref{thm:HorbezFP}, $\Out(\G)$ satisfies the Tits alternative.
	
	Now suppose $\Gamma$ is connected.
	First, if $\Gamma$ is the star of a vertex, with $\Gamma=\Gamma'\star K$, 
	then Proposition~\ref{prop:star_SES} and our inductive hypothesis implies that $\Out^1(\G)$ satisfies the Tits alternative. 
	
	If $\Gamma$ is connected and not a star, we apply the amalgamated projection map:
	$$P\colon \Out^1(\G_{\Gamma})\to \prod_{[v]\text{ maximal}}\Out(\G_{\lk[v]}).$$
	By induction, 
	$\Out(\G_{\lk[v]})$ satisfies the Tits Alternative for all maximal equivalence classes $[v]$. Further, by Theorem~\ref{thm:kerP_abelian}, the kernel of $P$ is abelian. Thus, $\Out^1(\G)$ is an extension of a subgroup of a product of groups which satisfy the Tits' alternative by an abelian group. Hence, $\Out^1(\G)$ satisfies the Tits Alternative.
\end{proof}

Indeed, \cite[Theorem A]{AntolinMinasyanTA}, which implies that every subgroup of $\G$ is either virtually abelian (and hence virtually polycyclic) or contains $F_2$. Using \cite[Theorem 6.1]{HorbezTA}, and the argument above we obtain the following, stronger, statement.

\begin{cor} \label{cor:polycyclic TA}
	Let $\G$ be a graph product of finitely generated abelian groups. 
	Then every subgroup of $\Out(\G)$ either contains $F_2$ or is virtually polycyclic.
\end{cor}

This is frequently called satisfying the Tits Alternative with respect to the class of virtually polycyclic groups. 
Furthermore, this implies that all abelian subgroups of $\Out(\G)$ are finitely generated.

\subsection{Residual Finiteness} We now show that $\Out(\G_\Gamma)$ is residually finite. The proof follows much of the same path as \cite[Theorem 10]{CV11}, though we refrain from using restriction maps, and instead use only the amalgamted projection homomorphism and Theorem~\ref{thm:kerP_abelian}. An alternative proof of this fact is due to Ferov \cite[Corollary 1.5]{Ferov:RF}.

\begin{proof}[Proof of Theorem~\ref{thms:OutGPRF}] We proceed by induction on the number of vertices of $\Gamma$. If $\Gamma$ has a single vertex, then $\Out(\G_\Gamma)$ is finite, and hence residually finite. 
	
	Now suppose that $\Gamma$ has more than one vertex.
	
	 Minasyan and Osin showed that the outer automorphism group of a residually finite group with infinitely many ends is residually finite \cite[Theorem 1.5]{MinasyanOsin:Aut}.
	Meanwhile, Green and, by other means, Hsu and Wise, proved that the graph product of residually finite groups is residually finite \cite[Corollary 5.4]{Green:graphprod}, \cite[Theorem 3.7]{HsuWise}.
	These two results, and the fact that $\Out(\Z_2\ast\Z_2)$ is finite, prove that when $\Gamma$ is disconnected, $\Out(\G_\Gamma)$ is residually finite. 
	
	If $\Gamma$ is the star of a vertex, we use Proposition~\ref{prop:star_SES} and the fact that a semi-direct product of residually finite groups is residually finite.
	
	Finally, suppose $\Gamma$ is connected and not a star. It is enough to show that $\Out^1(\G)$ is residually finite. We consider the amalgamated projection homomorphism:
	$$P\colon \Out^1(\G_{\Gamma})\to \prod_{[v] \text{ maximal}} \Out^1(\G_{\lk[v]}).$$
	
	By induction, for each maximal equivalence class $[v]$, $\Out^1(\G_{\lk[v]})$ is residually finite. Thus, if $\Phi\in \Out^1(\G_{\Gamma})$ is such that $P(\Phi)\neq 1$, then there is a finite quotient of $\Out^1(\G_{\lk[v]})$ for some maximal equivalence class $[v]$ where $\Phi$ survives.
	
	If $P(\Phi)=1$, then by Theorem~\ref{thm:kerP_abelian} we can write $\Phi=R_1\cdots R_s\pi_1\cdots\pi_k$, where $R_1,\ldots,R_s$ are powers of leaf-like transvections and $\pi_1,\ldots,\pi_k$ are (products of) partial conjugations with distinct multipliers $v_i$ and whose supports are unions of bridged $v_i$--components.
	Take such a product with $s+k$ minimal.
	
	If $s\neq 0$, then let $R_1=R^{v^k}_u$ be the first leaf transvection. Let $[w]$ be a maximal equivalence class in $\lk[v]$. 
	Then since $u\le v$ and they are not equivalent, $[w]\neq [u]$.
	Consider the factor map:
	$$E\colon \Out^1(\G_{\Gamma})\to \Out^1(\G_{\Gamma-[w]}).$$
	Note that such a homomorphism exists, by Lemma~\ref{lem:preservation of maximal stars}, since $[w]$ is maximal. 
 Given the structure of $\Phi$ as the product of transvections and partial conjugations above, 
	the image of $u$ under $E(\Phi)$ will be conjugate to an element of the form $uy$, where $y \in \grp{[v]}$ and is non-trivial.
	It follows therefore that $E(\Phi)(u)$ is not conjugate to $u$ (see for example \cite[Lemma 3.12]{FerovConjSep}). In particular $E(\Phi)$ is non-trivial, and, by induction, $\Phi$ survives in some finite quotient.
	
	If $s=0$, this can be dealt with as in \cite[Theorem 10]{CV11}. We have $\Phi=\pi_1\cdots\pi_k$, and since the kernel of $P$ is abelian we can group the partial conjugations together so $\Phi = \bar{\pi}_1 \cdots \bar{\pi}_{k'}$ where $\bar{\pi}_i$ is a product of partial conjugations with multipliers in $[v_i]$, and $[v_i]\neq[v_j]$ for $i\neq j$.
	Take $[w]$ a maximal equivalence class in $\lk [v_1]$, and use the same exclusion map $E$ as above.
	Since the deleted vertices are adjacent to $[v_1]$, the image of $\bar{\pi}_1$ remains a (non-inner) product of partial conjugations.
	Let $u$ be in the support of $\bar{\pi}_1$, such that $\bar{\pi}_1(u) = v_1^kx ux^{-1}v_1^{-k}$, with $k\neq 0$ and $x\in \grp{[v_1]\setminus \{v_1\}}$. 
	Then $\Phi(u) = gug^{-1}$ for some $g$ in which  the exponent sum of $v_1$ is equal to $k$.
	It follows that $E(\Phi)(u) \neq u$, completing the proof.
\end{proof} 

\subsection{Other potential applications}

In \cite{CV09, CV11} Charney and Vogtmann also used the inductive technique above to  prove other results for RAAGs which we do not extend here. For instance,
we do not investigate the virtual cohomological dimension. 
Since the kernel of the amalgamated projection homomorphism may have torsion, it is not clear that it can be used to prove that $\Out(\G)$ is virtually torsion-free. We  note, however, that following work of Carette \cite{Carette} we can deduce that $\Out(\G)$ is virtually torsion-free whenever $o(v)<\infty$ for all vertices $v$. The unresolved case is when some vertices have finite order and some have infinite order.

\section{Generating the Torelli group}\label{sec:torelli}

The standard representation of $\Aut(\G)$ is obtained by acting on the abelianization $\bar{\G}$ of $\G$:
$$\rho \colon \Aut(\G) \to \Aut(\bar{\G}).$$
The kernel of $\rho$ is denoted by $\IA_\G$, and is sometimes referred to as the Torelli subgroup of $\Aut(\G)$.
The aim of this section is to prove the following result, which gives Theorem~\ref{thms:Torelli fg}.

\begin{thm}\label{thm:Torelli gen set}
	Let $\G$ be the graph product of finitely generated abelian groups.
	Then the Torelli subgroup $\IA_\G$ is generated by the set of all partial conjugations and commutator transvections $R_u^{[v,w]} = [R_u^v,R_u^w]$, when $u\leq_\infty v,w$.
\end{thm}

We note that Proposition~\ref{prop:torelli gp finite} already proves Theorem~\ref{thm:Torelli gen set} in the case when all vertex groups are finite, since in this case there are no commutator transvections.
However this is a very special case, and, as we shall see, having infinite order vertices makes things considerably more complicated.

As discussed in the introduction, Theorem~\ref{thm:Torelli gen set} generalizes the result of Day for RAAGs \cite{Day_symplectic} (see also Wade \cite{WadeThesis}).
In proving it we often appeal to both sources for inspiration, particularly in finding necessary identities at certain steps of the proof.

The first step is to show that we may restrict to one of our finite-index subgroups without changing the kernel of the standard representation.

\begin{prop}\label{prop:Torelli in 1 inf}
	The Torelli group $\IA_\G$ is a subgroup of $\Aut^1_\infty(\G)$.
\end{prop}

\begin{proof} 
	We first show that $\IA_\G\le \Aut^0(\G)$.
Given $\phi\in \Aut(\G)$, we decompose $\phi=\alpha \phi_0$, where $\alpha$ is the identity or an automorphism induced by a graph symmetry which is not in $\Aut^0(\G)$, and $\phi_0\in\Aut^0(\G)$. 
	If $\phi\in \IA_\G$ then $\rho(\alpha) = \rho(\phi_0^{-1})$.
	We appeal to Proposition~\ref{prop:getting graph symmetries}.
	If $\alpha \neq 1$ then $\alpha \not\in\Aut^0(\G)$ and there is some equivalence class $X$ such that $\G_{\geq X} \neq \G_{\geq\alpha X}$.
	In particular, there is some $x\in X$ such that $\alpha(x)$ is not in $\geq${}$X$.
	However, $\phi_0^{-1} \in \Aut^0(\G)$, so by Proposition~\ref{prop:getting graph symmetries}, $\phi_0^{-1}(x) \in \G_{\geq X}$.
	This implies that $\rho(\alpha)(x) \neq \rho(\phi_0^{-1})(x)$, a contradiction.
		
	Now we show that $\IA_\G\le \Aut^1_\infty(\G)$.
	To do this, we use the homomorphism $D$ from \eqref{eq:defn D} in Section~\ref{sec:homom D}, whose kernel is $\Aut^1_\infty(\G)$ by Proposition~\ref{prop:kernel D}.
	 Recall that $D$ is a homomorphism constructed by bundling together homomorphisms $D_X$, for equivalence classes $X$ containing infinite order elements, and $D_0$:
	$$
	D := D_0 \times \prod_{X\in\cal{E}} D_X \colon \Aut^0(\G) \to \Aut(T) \times \prod_{X\in\cal{E}}  \Z_2 .
	$$	
	The fact that $\IA_\G \subseteq \ker D$ follows because each of the maps $D_X$ or $D_0$ involves the standard representation, or a restriction of it.	
\end{proof} 

We can write the abelianization of $\G$ as $\bar{\G} \cong \Z^n \times T$, where $T$ is a finite abelian group.
We denote the image of $R_u^v \in \Aut^1_\infty(\G)$ under $\rho$ by $E_u^v$.
Let $C_1,\ldots,C_k$ be the list of all $\leq_\infty$--equivalence classes in $\Gamma$ that contain infinite order vertices.

\begin{lem}\label{lem:standard rep image relators}
	The image of $\Aut^1_\infty(\G)$ under $\rho$ is isomorphic to
	$$E \rtimes (N\rtimes (\SL(n_1,\Z) \times \cdots \times \SL(n_k,\Z)))$$
	where $N$ is a finitely generated nilpotent group, and $E$ is a subgroup of $T^n$ generated by the image of transvections $E^v_u$, with $o(v)<\infty$, and $n_i = \abs{C_i}$, for $i=1,\ldots k$, are the sizes of the infinte-order equivalence classes of $\Gamma$.
	
	Furthermore, it has presentation with generators
	$$\left\{ {E_u^v} \colon u\leq_\infty v \right\}$$
	and relators
	\begin{enumerate}
		\item $[{E_u^v}, {E_x^y}]$ if $v\neq x$ and $u\neq y$;
		\item $[{E_v^w}, {E_u^v}] E_u^{w^{-1}}$ if $u\neq w$;
		\item $(E_u^v E_v^{u^{-1}} E_u^v)^4$ if $u,v \in C_i$ for some $i$;
		\item $(E_u^v E_v^{u^{-1}} E_u^v)^2 (E_u^v E_v^{u^{-1}} E_u^v E_v^u)^{-3}$ if $\{u,v\}=C_i$ for some $i$;
		\item $(E_u^v)^{o(v)}$ if $o(v)<\infty$.
	\end{enumerate}
\end{lem}

\begin{proof}
	Let $n=n_1+\cdots +n_k$. Seeing the image as the given semidirect product is possible by using the factor map from $\Aut(\bar{G})$ to $\Aut(\Z^n)$, by ignoring all finite order vertices.
	The kernel of this map is the subgroup generated by $E_u^v$ for $u\leq_\infty v$ and $o(v)<\infty$, which is isomorphic to a subgroup of $T^n$, which we denote $E$.
	The image is the subgroup generated by those $E_u^v$ for which $u\leq_\infty v$ and $o(v)=\infty$, giving the subgroup of $\SL(n,\Z)$.
	Each diagonal block $\SL(n_i,\Z)$ corresponds to the infinite order vertices in the one equivalence class.
	The subgroup $N$ is generated by those $E_u^v$ for which $u$ and $v$ are not in the same equivalence class.
	The fact that $N$ is nilpotent is immediate since it is isomorphic to a group of lower triangular unipotent matrices.
	It is easy to see that this map is split.
	
	To complete the proof of the Lemma we need to verify the given relators are sufficient.
	Firstly, there are sufficient relators for $E$, using those of the first and fifth types.
	From Day \cite[Proposition 3.10]{DayPeakReduction}, these relators are sufficient to give a presentation  for $N\rtimes (\SL(n_1,\Z) \times \cdots \times \SL(n_k,\Z))$ (c.f.{} the Steinberg relators).
	The semidirect product action of this group on $E$ is encoded within the relators of the first and second types, with $o(y),o(w)<\infty$, by Lemma~\ref{lem:relators}.
\end{proof}

Before we advance, we note the following simple criteria that helps determine when certain automorphisms commute.

\begin{lem}\label{lem:commuting criteria}
	Let $\alpha,\beta$ be partial conjugations or transvections with multipliers $a,b$ and supports   and $C,D$ respectively.
	Then $\alpha$ and $\beta$ commute if either
	\begin{itemize}
		\item $a=b$ and $C\cap D = \emptyset$;
		\item $(\{a\}\cup C)\cap (\{b\}\cup D) = \emptyset$;
		\item $[a,b]=1$ and $C=D$.
	\end{itemize}
\end{lem}

Recall our conventions: $\pi^v_C$ sends $z\in C$ to $vzv^{-1}$; and $[a,b] = aba^{-1}b^{-1}$. Further, we denote global conjugation by $v$ (sending each $z\in \Gamma$ to $vzv^{-1}$) as $\ad_v$.

\begin{lem}\label{lem:Torelli normal}
	Let $K$ be the subgroup of $\Aut^1_\infty(\G)$ generated by the set of all partial conjugations and all commutator transvections $R_u^{[x,y]} = [R_u^x,R_u^y]$ for $u \leq_\infty x,y$.
	Then $K$ is normal in $\Aut^1_\infty(\G)$.
\end{lem}

	To prove this, we need to show that the conjugate of each generator by a transvection, or its inverse, remains in $K$. We split the proof into two separate lemmas, for future reference. 
		
\begin{lem}\label{lem:conjugates of pc by transv}
	Suppose $v,x,y\in \Gamma$ with $x\leq_\infty y$, and $C$ is a connected component of $\Gamma \setminus \st(v)$. Then either $\pi^v_C$ and $R_x^{y}$ commute or $R_x^{y^{\mp1}} \pi^v_C R_x^{y^{\pm1}}$ is equal to:
	\begin{align}
	\tag{$\cal{R}$1}\label{equn:normal1} 
	&\pi^v_C R_x^{[v,y^{\mp 1}]} && \textrm{if $x\in C$ and $y\notin C\cup \st(v)$,}\\
	\tag{$\cal{R}$2}\label{equn:normal2} 
	&R_x^{[y^{\mp 1},v]} \pi^v_C && \textrm{if $x\notin C\cup \st(v)$ and $y\in C$,}\\
	&	\tag{$\cal{R}$3}\label{equn:normal3} 
	\pi^{y^{\mp1}}_C \pi^{x}_C   && \textrm{if $v=x$ and $y \notin C$,}\\
	&	\tag{$\cal{R}$4}\label{equn:normal4} 
	\pi^{x\m}_{C'}\pi^{y^{\pm1}}_{C'}\ad_x\ad_y^{\mp 1}   && \textrm{if $v=x$ and $y \in C$,}
	\end{align} 
	where $C' = \Gamma \setminus (\st(x)\cup C)$.
\end{lem}		

\begin{proof}	
	If $v\notin \{x,y\}$ and $x,y\notin C$, then $\pi^v_C$ and $R^y_x$ commute by Lemma~\ref{lem:commuting criteria}.
	If $v\not\in\{x,y\}$, but $x,y\in C\cup \st(v)$, then direct calculation shows that $R_x^y$ and $\pi^v_C$ commute.
	
	Thus we may assume $x$ and $y$ are in different connected components of $\Gamma\setminus\st(v)$, with one being in $C$.
	This implies $x$ and $y$ are not adjacent and since $\lk(x) \subseteq \st(y)$, we must have $\lk(x)\subseteq \st(v)$, and hence $x\leq_\infty v$.
	 If $x\in C$, in particular $C = \{x\}$ since $\lk(x) \subset \st(v)$, and manual computation verifies relation \eqref{equn:normal1} and \eqref{equn:normal2}.

	Suppose $v=x$. First assume $y \notin C$.  Since $\lk(x)\subseteq \st(y)$, $C$ is contained in the union of a connected component of $\Gamma\setminus\st(y)$ and $\st(y)$. So, $\pi^{y}_C$ is well-defined.
	Then manual calculations give relation \eqref{equn:normal3}.

	Now assume $y \in C$.
	Define $C' = \Gamma \setminus (\st(x)\cup C)$.
	Then \ref{equn:normal3} applies for $\pi_{C'}^x$ in place of $\pi_C^x$.
	Furthermore, $\pi^x_C = \pi^{x\m}_{C'} \ad_x$, where $\ad_x$ is the inner automorphism by $x$,
	and $\ad_x R_x^{y\m} = R_x^{y\m} \ad_x \ad_y$ (which can be verified by direct computation).
	Plugging this (and the fact that $\ad_y$ commutes with $R_x^y$) into relation \eqref{equn:normal3} and rearranging gives relation \eqref{equn:normal4}.

	Suppose $v=y$. If $x\notin C$, then $R_x^v$ and $\pi^v_C$ commute by Lemma~\ref{lem:commuting criteria}.
	If $x\in C$ then $C = \{x\}$ and again Lemma~\ref{lem:commuting criteria} gives us that the automorphisms commute.
\end{proof}

\begin{lem}\label{lem:conjugates of ct by transv}
	Suppose $x,y,z,u,v\in \Gamma$ with $x\leq_\infty y,z$, and $u\leq_\infty v$. 
	Then either $R_x^{[y,z]}$ and $R_u^{v}$ commute or $R_u^{v} R_x^{[y,z]} R_u^{v^{-1}}$ is equal to:
\begin{align}
	\tag{$\cal{R}$5}\label{equn:normal5}  
	&\pi_{\{u\}}^{v^{-1}} R_u^{[y,z]} \pi_{\{u\}}^{v}  && \textrm{if $u=x$;}\\
	\tag{$\cal{R}$6}\label{equn:normal6} 
	&\pi_{\{u\}}^{v^{-1}} R_u^{[z,y]} \pi_{\{u\}}^{v} R_v^{[y,z]} && \textrm{if $x = v$ and $u\notin \{y,z\}$;}\\
	\tag{$\cal{R}$7}	\label{equn:normal7} 
	& \alpha 
	R^{[u,z]}_v
	\beta 
	&& \textrm{if $x=v$ and $y=u$;}\\
	\tag{$\cal{R}$8}	\label{equn:normal8} 
	& \gamma 
	R^{[y,u]}_v 
	\delta 
	&& \textrm{if $x=v$ and $z=u$;}\\
	\tag{$\cal{R}$9}	\label{equn:normal9} 
	&\pi_{\{x\}}^{u^{-1}} R_x^{[v,z]} \pi_{\{x\}}^u R_x^{[u,z]} && \textrm{if $y=u$ and $[v,z]\ne 1$;}\\
	\tag{$\cal{R}$10}	\label{equn:normal10} 
	&R_x^{[y,u]} \pi_{\{x\}}^{u^{-1}} R_x^{[y,v]} \pi_{\{x\}}^u && \textrm{if $z=u$ and $[v,y]\ne 1$.}
\end{align}
In \eqref{equn:normal7}, 
$\alpha = \pi^v_Z \pi^u_Z \pi^z_{\{u\}} R^{[z,v]}_u \pi^u_{\{v\}}$, 
$\beta =  \pi^{z^{-1}}_{\{v\}} \pi^{u^{-1}}_{\{v\}} \pi^{u^{-1}}_Z \pi^{v^{-1}}_Z$, 
and $Z$ is the connected component of $\Gamma\setminus \st(u)$ and $\Gamma\setminus\st(v)$ that contains $z$. 

In \eqref{equn:normal8}, 
$\gamma = \pi^v_Y \pi^u_Y \pi^u_{\{v\}} \pi^y_{\{v\}} $, 
$\delta = \pi^{u\m}_{\{v\}} 	R^{[v,y]}_u\pi^{y^{-1}}_{\{u\}} \pi^{u^{-1}}_{Y}  \pi^{v^{-1}}_Y$, and 
$Y$ is the connected component of $\Gamma\setminus \st(u)$ and $\Gamma\setminus\st(v)$ that contains $y$.
\end{lem}	
	
\begin{proof}
	If $x,y,z,u,v$ are all distinct then $R_u^v$ and $R_x^{[y,z]}$ commute by Lemma~\ref{lem:commuting criteria}.
	
	If $x=u$, then $\pi^v_{\{u\}}$ exists, and we have relation \eqref{equn:normal5}.
	
	If $x=v$ and $u\notin \{y,z\}$, then $u\le_\infty y,z$. Direct calculations yield relations \eqref{equn:normal6}. 
	
	Now assume $x=v$ and $u=y$.
	Then $[u]_\infty = [v]_\infty$ and $u,v\leq_\infty z$, so partial conjugations $\pi^z_{\{u\}},\pi^z_{\{v\}},\pi^v_{\{u\}},\pi^u_{\{v\}}$ exist,
	and furthermore the connected component of $\Gamma \setminus \st(u)$ containing $z$ is equal to the corresponding component of $\Gamma \setminus \st(v)$, which we denote by $Z$.
	The patient reader may verify that relation \eqref{equn:normal7} holds, and that equation \eqref{equn:normal8} follows from \eqref{equn:normal7} by realizing that $R_v^{[y,u]}$ is the inverse of $R_v^{[u,z]}$.
	
	We are left with $x\notin\{u,v\}$.
	If $u \notin \{y,z\}$ then $R_u^v$ and $R_v^{[y,z]}$ commute by Lemma~\ref{lem:commuting criteria}.
	So assume $u=y$.
	If $[v,z] = 1$, then $R_u^v$ and $R_x^{[y,z]}$ commute, else 
	$x\le_\infty u$ implies that $\pi^u_{\{x\}}$ exists and relation \eqref{equn:normal9} holds.
	Finally, relation \eqref{equn:normal10} follows from \eqref{equn:normal9} by using that $R_x^{[z,y]}$ is the inverse of $R_x^{[y,z]}$.
\end{proof}

\begin{proof}[Proof of Lemma~\ref{lem:Torelli normal}]
	Lemma~\ref{lem:conjugates of pc by transv} shows that any partial conjugation in $K$, when conjugated by a transvection, remains in $K$.
	
	Lemma~\ref{lem:conjugates of ct by transv} shows that conjugates of the form $R_u^v R_x^{[y,z]} R_u^{v\m}$ are in $K$.
	The other conjugates, $R_u^{v\m} R_x^{[y,z]} R_u^v$, are as follows below.
	Cases \eqref{equn:normal5a} and \eqref{equn:normal6a} can be quickly checked by direct computation.
	Cases \eqref{equn:normal7a}--\eqref{equn:normal10a} follow from \eqref{equn:normal7} and \eqref{equn:normal8} by algebraic manipulation.
	Each of $\alpha,\beta,\gamma,\delta$ are as in Lemma~\ref{lem:conjugates of ct by transv}.

		\begin{align}
			\tag{$\cal{R}$5$^\prime$}\label{equn:normal5a} 
			&\pi_{\{u\}}^{v} R_u^{[y,z]} \pi_{\{u\}}^{v\m} 
			&& \textrm{if $x=u$;}\\
			\tag{$\cal{R}$6$^\prime$}\label{equn:normal6a} 
			&R^{[y,z]}_u R^{[y,z]}_v 
			&& \textrm{if $x = v$ and $u\notin \{y,z\}$;}\\
			\tag{$\cal{R}$7$^\prime$}\label{equn:normal7a} 
			&R^{v^{-1}}_u  
			\alpha\m 
			R^v_u \ R_v^{[u,z]} \ R^{v^{-1}}_u
			\beta\m 
			R^v_u 
			&& \textrm{if $x=v$ and $y=u$;}\\
			\tag{$\cal{R}$8$^\prime$}\label{equn:normal8a} 
			&R^{v^{-1}}_u  
			\gamma\m 
			 R^{v}_u \ R_v^{[y,u]}	\ R^{v\m}_u
			 \delta\m 
			  R^v_u 
			  && \textrm{if $x=v$ and $z=u$;}\\
			 \tag{$\cal{R}$9$^\prime$}\label{equn:normal9a}
			&R^{v^{-1}}_u  \pi_{\{x\}}^{u^{-1}} R_x^{[z,v]} \pi_{\{x\}}^u  R_u^v R_x^{[u,z]}
			 && \textrm{if $y=u$ and $[v,z]\ne 1$;}\\
			 \tag{$\cal{R}$10$^\prime$}\label{equn:normal10a}
			 & R_x^{[y,v]} R^{v^{-1}}_u \pi_{\{x\}}^{u^{-1}} R_x^{[v,y]} \pi_{\{x\}}^u R_u^v 
			  && \textrm{if $z=u$ and $[v,y]\ne 1$.}
		\end{align}
	It is immediate that in cases \eqref{equn:normal5a} and \eqref{equn:normal6a} that $R_u^{v\m} R_x^{[y,z]} R_u^v \in K$.
	Cases \eqref{equn:normal7a} and \eqref{equn:normal8a} follow from Lemma~\ref{lem:conjugates of pc by transv} and case \eqref{equn:normal5a}.
	Cases \eqref{equn:normal9a} and \eqref{equn:normal10a} also follow from Lemma~\ref{lem:conjugates of pc by transv} but also require that $R_u^v$ commutes with $R_x^{[z,v]}$ and $R_x^{[v,y]}$ by Lemma~\ref{lem:conjugates of ct by transv}.
\end{proof}

We now prove Theorem~\ref{thm:Torelli gen set} by showing that the normal subgroup $K$ is equal to the kernel of $\rho$. 

\begin{proof}[Proof of Theorem~\ref{thm:Torelli gen set}]	
	Let $\rho^1_\infty \colon \Aut^1_\infty (\G) \to \Aut(\bar{\G})$ be the restriction of the standard representation.
	It is not hard to see that $K$ is contained in $\ker \rho^1_\infty$.
 	We claim that $K=\ker\rho^1_\infty$, and hence is equal to $\IA_\G$ by Proposition~\ref{prop:Torelli in 1 inf}. 
	As $K$ is normal in $\Aut^1_\infty (\G)$ by Lemma~\ref{lem:Torelli normal}, 
	it is enough to verify that each of the relators for the image of $\rho^1_\infty$ given by Lemma~\ref{lem:standard rep image relators}
	can be seen as the image of an element from $K$.
	
	We use the obvious lift of $E_u^v$ to $R_u^v$.
	Then the 5 relators in Lemma~\ref{lem:standard rep image relators} lift to elements of $K$ as follows.
	
	The relators $[E_u^v,E_x^y]$, when $v\neq x$ and $u\neq y$, lift to 
	$$[{R_u^v}, {R_x^y}] = \begin{cases} 
	R_x^{[v,y]} & \textrm{if $u=x$ and $[v,y] \neq 1$,}\\
	1 & \textrm{otherwise}.
	\end{cases}
	$$
	
	The relators $[E_u^v,E_v^w]E_u^{w^{-1}}$, when $u\neq w$, lift to 
	$$[{R_v^w}, {R_u^v}] R_u^{w^{-1}} = \begin{cases}
	 R_u^{[v,w]} & \textrm{if $[v,w]\neq 1$,}\\
	 1 & \textrm{otherwise.}
	\end{cases}$$
	
	As in the RAAG case, \cite[Pg.{} 61]{WadeThesis}, the relators $(E_u^v E_v^{u^{-1}} E_u^v)^4$, when $u,v \in C_i$ for some $i$, lift to  
	$$(R_u^v R_v^{u^{-1}} R_u^v)^4 = \begin{cases}
	\pi^u_{\{v\}} \pi^v_{\{u\}}\pi^{u^{-1}}_{\{v\}} \pi^{v^{-1}}_{\{u\}} & \textrm{if $[u,v] \neq 1$}\\
	1 & \textrm{otherwise.}
	\end{cases}$$
	
	When $\{u,v\}=C_i$ for some $i$, the relators $(E_u^v E_v^{u^{-1}} E_u^v)^2 (E_u^v E_v^{u^{-1}} E_u^v E_v^u)^{-3}$ always lift to the identity, whether $u,v$ commute or not. 
	
	Finally, the relators $(E_u^v)^{o(v)}$, when $o(v)<\infty$, lift to $(R_u^v)^{o(v)} = 1$.
\end{proof}

\section{A short exact sequence when there is no free SIL}\label{sec:noSIL}

The aim of this section is to prove Theorem~\ref{thms:no SIL nil} (see Theorem~\ref{thm:no SIL nil class}), which generalizes \cite[Theorem 1.3, Proposition 2.11]{Day_solvable}. It gives conditions on $(\Gamma, o)$ that determine precisely when $\Out(\G)$ contains a nonabelian free subgroup. When this criterion is satisfied, we get two transvections or two or three partial conjugations that generate a virtually free group. The transvections arise when there is a suitable $\leq_\infty$--equivalence class, while the existence of the partial conjugations requires a sufficiently complex SIL (a ``free'' SIL). To show the other half of the dichotomy, it is important to understand the structure of $\Out(\G)$ when there is no free SIL, and this can be described by the short exact sequence of Theorem~\ref{thms:no SIL ses} (see Theorem~\ref{thm:no SIL SES}), which generalizes \cite[Theorem 2]{GuirardelSale-vastness}. Proving Theorem~\ref{thm:no SIL SES} takes us most of the way to proving Theorem~\ref{thm:no SIL nil class}.

When $\Out(\G)$ does not contain a nonabelian free subgroup we show it is nilpotent. In order to understand the nilpotency class we make the following definition.

\begin{defn}
	Let $(\Gamma,o)$ be a finite labeled graph.
	Given a vertex $v$ its \emph{$\infty$--depth}  is the largest integer $c$ such that either:
	
\begin{enumerate}
	\renewcommand{\theenumi}{$\cal{D}$\arabic{enumi}}
	\item\label{depth:tr only} there is a set of vertices $v_1,\ldots,v_c$ in distinct $\leq_\infty$--equivalence classes of $\Gamma$ such that:
	\begin{itemize}
		\item $v_1 \leq_\infty v_2 \leq_\infty \cdots \leq_\infty v_c=v.$
	\end{itemize}
	\item\label{depth:tr and pc} there is a set of vertices $v_1, \ldots, v_{c-1}$ in  distinct $\leq_\infty$--equivalence classes of $\Gamma$ such that:	
	\begin{itemize}
		\item $v_1\leq_\infty v_2 \leq_\infty \cdots \leq_\infty v_{c-1}=v$
		\item $\Gamma\setminus \st(v_1)$ contains two components $C_1, C_2\not\subseteq \st(v)$.
	\end{itemize}
\end{enumerate}

	The \emph{$\infty$--depth of the labeled graph $(\Gamma,o)$} is defined to be the largest $\infty$--depth of any vertex in $\Gamma$ whose order is not equal to 2.
	If every vertex has order 2, then $(\Gamma,o)$ has $\infty$--depth equal to 1.
\end{defn}

A complete statement of Theorem~\ref{thms:no SIL nil} is the following.

\begin{thm}\label{thm:no SIL nil class}
	Let $(\Gamma,o)$ be a finite labeled graph and $\G = \G(\Gamma,o)$.
	Then $\Out(\G)$ contains a nonabelian free subgroup if and only if $(\Gamma,o)$ contains either
	\begin{itemize}
		\item a $\leq_\infty$--equivalence class of size at least 2;
		\item a non-Coxeter SIL;
		\item a STIL;
		\item an FSIL.
	\end{itemize}
	Otherwise, $\Out(\G)$ contains a finite-index subgroup that is nilpotent of class equal to the $\infty$--depth of $(\Gamma,o)$.
\end{thm}

One direction of Theorem~\ref{thm:no SIL nil class} is not hard. The existence of the specified graphical features naturally gives rise to a free subgroup. 

Firstly, if $u$ and $v$ are distinct  vertices in the same $\le_\infty$--equivalence class, then they are both infinite order, and $(R^u_v)^2$ and $(R^v_u)^2$ generate a subgroup isomorphic to $F_2$.

If $(x,y\mids z)$ forms a non-Coxeter SIL, and $C$ is the component of $\Gamma\setminus(\lk(x)\cap \lk(y))$ containing $z$, then $\gspan{\pi^x_C, \pi^y_C}\cong \Z_{o(x)}\ast \Z_{o(y)}$, where we interpret $\Z_\infty=\Z$. This group is virtually non-abelian free, since either $o(x)\ge 3$ or $o(y)\ge 3$, or at least one of $x$ and $y$ has infinite order.
 Finally, if $(\Gamma,o)$ contains a STIL or FSIL, the three relevant partial conjugations in either case generate a subgroup isomorphic to the free product of three cyclic groups, which is virtually nonabelian free.

The converse, and indeed the ``otherwise'' statement, is a consequence of Theorem~\ref{thms:no SIL ses}. A version including the nilpotency class is given below.

\begin{thm}\label{thm:no SIL SES}
	Suppose $(\Gamma,o)$ does not contain either
	\begin{itemize}
		\item a non-Coxeter SIL,
		\item a STIL,
		\item an FSIL.
	\end{itemize}
	Then we have a short-exact sequence
	$$1 \to P \to \Out^1_\infty(\G) \to \prod \SL(n_i,\Z) \to 1$$
	where $P$ is finitely generated and virtually nilpotent of class equal to the $\infty$--depth of $(\Gamma,o)$, and the values of $n_i$ are the sizes of the $\leq_\infty$--equivalence classes.
\end{thm}

To construct the short-exact sequence we begin by taking the standard representation $\rho$ of $\Out^1_\infty(\G)$.
The image is described by Lemma~\ref{lem:standard rep image relators}, and admits as a quotient a product of special linear groups.
We define $\mu$ to be the composition of the standard representation with this quotient map, giving:
$$\mu \colon \Out^1_\infty(\G) \to \prod \SL(n_i , \Z).$$
The values $n_i$ correspond to the sizes of the $\leq_\infty$--equivalence classes containing more than one vertex.

\begin{rem} 
It's not hard to see that $\mu$ is a split surjection  when $(\Gamma, o)$ does not contain a non-abelian equivalence class of size at least 3, though that is immaterial in the proof of either Theorem~\ref{thm:no SIL nil class} or ~\ref{thm:no SIL SES}.
\end{rem}

The next lemma shows that the kernel of $\mu$ is generated by a subset of the standard generating set for $\Out(\G)$, whether $(\Gamma,o)$ contains any of the free SIL structures or not.

\begin{lem}\label{lem:ker mu gen set}
	The kernel of $\mu$ is generated by the set consisting of:
	\begin{itemize}
		\item  all partial conjugations, 
		\item  all commutator transvections $R_u^{[x,y]}$ for $u\leq_\infty x,y$,
		\item all transvections $R_u^v$ such that $u \leq_\infty v$, and $u \not\sim v$.
	\end{itemize}
Moreover, if $(\Gamma,o)$ satisfies the conditions of Theorem~\ref{thm:no SIL SES} then we may remove commutator transvections from the list of generators.
\end{lem}

\begin{proof}[Proof of Lemma~\ref{lem:ker mu gen set}]
	The homomorphism $\mu$ is the composition of two maps, $\rho$ (the standard representation) and a projection map $\pi$, defined explicitly below.
	We analyze the kernels of $\rho$ and $\pi$ separately.
	
	As in Lemma~\ref{lem:standard rep image relators}, we denote $E_u^{v} = \rho(R_u^{v})$, and have
	\begin{align*} \rho\left(\Out^1_\infty(\G)\right)  \cong E & \rtimes\left( N \rtimes \left( \SL(n_1,\Z) \times \cdots \times \SL(n_k , \Z) \right) \right)\\
	& \le T^n\rtimes\left( N \rtimes \left( \SL(n_1,\Z) \times \cdots \times \SL(n_k , \Z) \right) \right).\end{align*}
	The subgroup $E$ is generated by $E_u^v$, where $u\leq_\infty v$ and $o(v)<\infty$,
	and $N$ is generated by $E_u^v$ for $u\leq_\infty v$, $o(v) = \infty$ and $u\not\sim v$.
	
The homomorphism $\pi$ is the projection:
	$$\pi\colon  T^n\rtimes\left( N \rtimes \left( \SL(n_1,\Z) \times \cdots \times \SL(n_k , \Z) \right) \right) \to \left( \SL(n_1,\Z) \times \cdots \times \SL(n_k , \Z) \right).$$
	The kernel of $\pi$ in $\rho(\Out^1_\infty(\G))$ is generated by   $E_u^v$ for $u\leq_\infty v$, $o(v) = \infty$ and $u\not\sim v$.
	Thus, combining this with the generating set for $\IA_\G= \ker \rho$ given by Theorem~\ref{thm:Torelli gen set}, $\ker\mu$ is generated by the above transvections, all partial conjugations, and all commutator transvections.
	
	 To prove the ``moreover'' statement, consider $R^{[u,v]}_x\in\ker \mu$.
	Since there is no free SIL, if $u\sim x$, and $u, x$ are non-adjacent, then $(u,x\mids v)$ is a non-Coxeter SIL. Thus, $u, x$ must be adjacent. In this case though, since $\lk(x)\subseteq \st(v)$, we must have that $u, v$ commute, and so the commutator transvection would be trivial. Thus, if $R_x^{[u,v]}$ is nontrivial, then $R^u_x,R^v_x\in\ker\mu$, and the commutator transvections are products of (simple) transvections in our generating set, and can therefore be removed.
\end{proof}

In the next proposition we (expand and) refine the generating set given by Lemma~\ref{lem:ker mu gen set} to give a generating set for a finite-index subgroup of $\ker \mu$ which we will later prove is  nilpotent. 

\begin{prop}\label{prop:calA finite index}
	Let $\cal{A}$ be the subgroup of $\ker\mu$ generated by the following elements:
	\begin{enumerate}
		 \renewcommand{\theenumi}{$\cal{A}$\arabic{enumi}}
		\item\label{gen:pc} all partial conjugations $\pi^v_C$ with $o(v) \neq 2$,
		\item\label{gen:cpc} all commutator partial conjugations $\pi^{[v,w]}_C$ with $o(v)=o(w)=2$,
		\item\label{gen:tr} all transvections $R_x^y$ with $x\leq_\infty y$, $x\not\sim y$, and $o(y)\neq 2$,
		\item\label{gen:ctr} all commutator transvections $R_x^{[y,z]}$, with $x\leq_\infty y,z$ and $o(y)=o(z) = 2$.
	\end{enumerate}
 If $(\Gamma,o)$ does not contain a non-Coxeter SIL, a STIL, or an FSIL, 
then $\cal{A}$  has finite index in $\ker \mu$.
\end{prop}

Before proving the proposition, we investigate when automorphisms from this list commute, and see some of the graphical consequences if they do not. 
First though, we note that the list in Proposition~\ref{prop:calA finite index} includes all possible commutator transvections in the absence of a non-Coxeter SIL.

Note that in Lemmas~\ref{lem:pc and transv not commute gives SIL}, \ref{lem:pc and comm transv not commute STIL}, and \ref{lem:no SIL comm transv commute with transv} there is no assumption on the orders of the vertices. By restricting orders, certain cases can be excluded from consideration.

\begin{lem}\label{lem:comm transv gives SIL}
	Suppose $R_x^{[y,z]}$ is a non-trivial commutator transvection in $\Out(\G)$ (equivalently suppose $R_x^y$ and $R_x^z$ do not commute).
	
	Then $(y,z \mids x)$ is a SIL and $\lk(x)\subseteq \lk(y)\cap\lk(z)$.
\end{lem}

\begin{proof}
	Firstly, note that if we have a commutator transvection $R_x^{[y,z]}$ then we require $\lk(x)\subseteq \st(y)\cap\st(z)$. 
	We must have $[y,z] \neq 1$, so $\st(y)\cap\st(z) = \lk(y)\cap\lk(z)$,
	and furthermore  if $x\in\lk(y)\cap\lk(z)$, then $y\in\lk(x)\subseteq\lk(y)$, a contradiction. 
	Thus $\{x\}$ is a connected component of $\Gamma\setminus(\lk(y)\cap\lk(z))$, forming a SIL.	
\end{proof}

\begin{lem}\label{lem:pc and transv not commute gives SIL}
	Suppose 
	$R_x^y$ and $\pi^v_C$ are outer automorphisms of $\G$ that do not commute.
	Then either 
	\begin{itemize}
		\item $(v,y\mids x)$ is a SIL and $\lk(x)\subseteq \lk(v)\cap\lk(y)$,
		\item $v=x$ and $[R^y_x , \pi^x_C] = \pi^y_C$.
	\end{itemize}
\end{lem}

\begin{proof}
Lemma~\ref{lem:conjugates of pc by transv} tells us what happens when $R_x^y$ and $\pi^v_C$ do not commute, in the case when $C$ is connected.
Relations~\eqref{equn:normal1} and~\eqref{equn:normal2} therein both imply there is a SIL $(v,y\mids x)$ and $\lk(x)\subseteq \lk(v)\cap \lk(y)$ by Lemma~\ref{lem:comm transv gives SIL}.
Meanwhile relations~\eqref{equn:normal3} and~\eqref{equn:normal4} both lead to the given commutator in $\Out(\G)$.

It is not hard to generalize this to the case when $C$ is not a single connected component.
Proceeding by induction, if $C$ is a disjoint union $C_1\cup C_2$ with $C_2$ a connected component, then
$$[R_x^y , \pi^x_{C_1} \pi^x_{C_2}] = [R_x^y , \pi^x_{C_1}] \pi^x_{C_1} [R_x^y , \pi^x_{C_2}] \pi^x_{C_2}\pi^{x\m}_{C_1}\pi^{x\m}_{C_2} = \pi^y_{C_1}  \pi^x_{C_1} \pi^y_{C_2} \pi^x_{C_2}\pi^{x\m}_{C_1}\pi^{x\m}_{C_2}.$$
Since $C_1\cap C_2\neq\emptyset$, and $\pi^x_{C_1}$ and $\pi^x_{C_2}$ commute, by swapping $C_1$ with $C_2$ above, we can assume that $y\not\in C_1$. Thus, $[\pi^y_{C_1},\pi^x_{C_2}]=1$ by Lemma~\ref{lem:commuting criteria}, and the above reduces to $\pi^y_{C_1}  \pi^y_{C_2}$ as required.
\end{proof}

\begin{lem}\label{lem:pc and comm transv not commute STIL}
	Suppose 
	$R_x^{[y,z]}$  and $\pi^v_C$ are outer automorphisms of $\G$ that do not commute.
	Then either
	\begin{itemize}
		\item $v,x,y,z$ are distinct and $(v,y,z \mids x)$ is a STIL,
		\item $v=x$ and $\{x,y,z\}$ is an FSIL,
		\item $v=y$, $C=\{x\}$ and 
		$[R_x^{[y,z]},\pi^y_C] = R_x^{[y,z]}R_x^{[y^{-1},z]}$,
		\item  $v=z$, $C=\{x\}$ and
		$[R_x^{[y,z]},\pi^z_C] = R_x^{[y,z]}R_x^{[y,z^{-1}]}$.
	\end{itemize}
\end{lem}

\begin{proof} 
	First suppose that $v,x,y,z$ are distinct vertices. By Lemma~\ref{lem:comm transv gives SIL} we have a SIL $(y,z\mids x)$.
		Since $\pi^v_C$ cannot commute with both $R_x^y$ and $R_x^z$, Lemma~\ref{lem:pc and transv not commute gives SIL} implies either $(v,y\mids x)$ or $(v,z\mids x)$ is a SIL Thus Lemma~\ref{lem:overlapping SILs give STIL} implies $(v,y,z\mids x)$ is a STIL.

	Now we assume $v,x,y,z$ are not distinct.
	Suppose $v=x$.
	By Lemma~\ref{lem:comm transv gives SIL},  $(y,z\mids x)$ is a SIL and $\lk(x)\subseteq \lk(y)\cap\lk(z)$. Suppose that $\st(x)$ separates $y$ and $z$. Then $\lk(x)=\lk(y)\cap\lk(z)$, and there is no path from $y$ to $z$, $y$ to $x$, or $x$ to $z$ outside of $\lk(x)=\lk(y)\cap\lk(z)$. This implies that $\{x,y,z\}$ is an FSIL.
	
	 Thus, if this is not an FSIL, then $y$ and $z$ must be in the same component of $\Gamma\setminus\st(x)$.
	Without loss of generality replacing $x$ by $x^{-1}$, we may thus assume neither is in $C$.
	Direct calculation then shows that $[R_x^{[y,z]} , \pi^x_C] = 1$.	
	
	When $v=y$ or $z$, if $C \ne \{x\}$ then by Lemma~\ref{lem:commuting criteria} $\pi^y_C$ and $\pi^z_C$ commute with $R_x^{[y,z]}$.
	We can therefore assume $C = \{x\}$.
	Calculation then yields the given relations.
\end{proof}

\begin{lem}\label{lem:no SIL comm transv commute with transv}
	Suppose $R_x^{[y,z]}$ and $R_u^v$ are outer automorphisms of $\G$ that do not commute.
	Then either
	\begin{itemize}
		\item $v,x,y,z$ are distinct and $(v,y,z \mids x)$ is a STIL,
		\item $u,x,y,z$ are distinct and $(x,y,z\mids u)$ is a STIL,
		\item $u=x$, $v=y$ and $[R_x^{[y,z]},R_x^y]=  [ R_x^{[y,z]}, \pi^{y\m}_{\{x\}}]$,
		\item $u=x$, $v=z$ and $[R_x^{[y,z]},R_x^z]= [ R_x^{[y,z]}, \pi^{z\m}_{\{x\}}]$,
		\item $u=y$ or $z$, $v=x$ and $\{x,y,z\}$ is an FSIL,
	\end{itemize}
\end{lem}

\begin{proof}
	First note that since $R^{[y,z]}_x$ is non-trivial, $(y,z\mids x)$ is a SIL and $\lk(x)\subseteq \lk(y)\cap \lk(z)$. The proof proceeds in several cases.

	\textbf {Case 1.} First suppose $v,x,y,z$ are distinct vertices.	
	Since the automorphisms do not commute, by Lemma~\ref{lem:relators} we must have that $u \in \{x,y,z\}$.
	
	First suppose $u=x$. 
	Since $R_x^v$ cannot commute with both $R_x^y$ and $R_x^z$, we must have either a SIL $(v,y\mids x)$ or $(v,z\mids x)$.
	In either case Lemma~\ref{lem:overlapping SILs give STIL} gives a STIL $(v,y,z\mids x)$.
	Hence $(v,y,z\mids x)$ is a STIL.
		
	If $u=y$ then we  have $\lk(x)\subseteq \lk(y)\cap\lk(z) \subseteq \st(v)$ since $y\leq v$.
	Direct calculation shows that $R_x^{[y,z]}$ and $R_y^v$ commute whenever $v$ commutes with both $y$ and $z$.
	We therefore again get a SIL $(v,y\mids x)$ or $(v,z\mids x)$ and hence a STIL by Lemma~\ref{lem:overlapping SILs give STIL}.
	When $u=z$ we similarly get a STIL.

	\textbf{Case 2.} Now suppose $u,x,y,z$ are  distinct.
	As above, in order for the automorphisms to not  commute, we need $v\in \{x,y,z\}$.
	
	If $v=x$ we have $\lk(u) \subseteq \st(x)$ and $\lk(x)\subseteq \lk(y)\cap\lk(z)$.
	So if $u$ and $x$ are not adjacent we have $\lk(u) \subseteq \lk(x)\cap\lk(y)\cap\lk(z)$, and since $(y,z\mids x)$ forms a SIL---in particular $\grp{x,y,z}$ is not virtually abelian--we have a STIL $(x,y,z\mids u)$.
	So assume that $u\in \lk(x)$. Then also $u\in \lk(y)\cap\lk(z)$. But then $y,z \in \lk(u) \subset \st(x)$, contradicting that fact that $(y,z\mids x)$ is a SIL.
	
	If $v=y$ then $R_u^y$ commutes with both $R_x^y$ and $R_x^z$ by Lemma~\ref{lem:relators}, and hence $R^{[y,z]}_x$ and $R^{y}_x$ commute. They again commute when $v=z$.
	
	\textbf{Case 3.} Finally, suppose that $u,v\in\{x,y,z\}$. 
	Suppose $u=y$ and $v=x$.
	Since $y\leq x$ we have $\lk(y)=\lk(x) \subseteq \lk(z)$, implying that $\{x,y,z\}$ is an FSIL.
	When $u=z$ instead, the argument is similar.
The remaining cases follow from Lemma~\ref{lem:conjugates of ct by transv}.
	\end{proof}

The last combination we consider are commutators of  commutator partial conjugations with either partial conjugations or another commutator partial conjugation. This case was resolved in the authors' previous paper \cite{SaleSusse}, noting the proofs of the two lemmas carry through to this more general situation.

\begin{lem}[{\cite[Lemmas 2.8 and 2.10]{SaleSusse}}]\label{lem:SS17}
	Suppose $\Gamma$ does not contain a STIL or FSIL. Let $u,v,w,x,y$ be vertices of $\Gamma$ with $o(v)=o(w)=o(x)=o(y)=2$.
	Let $A$ be a  connected component of $\Gamma\setminus \st(u)$, $B$ a shared component of  $\Gamma\setminus \st(x)$ and $\Gamma\setminus \st(y)$, and $C$ a shared component of  $\Gamma\setminus \st(v)$ and $\Gamma\setminus \st(w)$.
	
	Then $\pi^u_A$ and $\pi^{[v,w]}_C$ commute, and $\pi^{[x,y]}_B$ and $\pi^{[v,w]}_C$ commute.
\end{lem}

We are now in a position to prove Proposition~\ref{prop:calA finite index}. We will use the relations above to show that every element of $\ker\mu$ can be written as a product so that all generators whose multipliers have order 2 are pushed all the way to the right,
and ultimately that $\cal A$ has finite index.

\begin{proof}[Proof of Proposition~\ref{prop:calA finite index}]
	Start with an outer automorphism $\Phi$ in $\ker\mu$, and express it as a word $w = \Phi_1 \cdots \Phi_k$ on the generating set from Lemma~\ref{lem:ker mu gen set}. 
	Let $H$ be the subgroup of $\ker\mu$ generated by all transvections and partial conjugations with multiplier of order $2$.
	First we show that we can write $\Phi =\alpha h$ so that $\alpha \in \cal{A}$ and $h\in H$.
	To do this, we shuffle letters $\Phi_i$ that are in $H$ to the right of the given word for $\Phi$.	
	
	Suppose $\Phi_i$ is the right-most letter with multiplier of order 2, and assume first that it is a transvection $\Phi_i = R_x^{y}$.
	We begin by moving $R_x^{y}$ to the right past $\Phi_{i+1}$.
	First note that $R_x^y$ must commute with any transvection of the form $R_x^z$ with $o(z)\neq 2$ since by Lemma~\ref{lem:comm transv gives SIL} we would otherwise get a non-Coxeter SIL $(y,z\mids x)$.
	As for other transvections, by Lemma~\ref{lem:relators} $R_x^y$ will commute with all except any with multiplier $x$ (note that transvections in $\Out^1_\infty(\G)$ have infinite order supporting vertex, which excludes the other possibility from Lemma~\ref{lem:relators}).
	So, if $\Phi_{i+1} = R_z^{x ^{\pm1}}$, then:
	\begin{equation}\label{eq:shuffle 1}
	\Phi_i \Phi_{i+1} = \Phi_{i+1} R_z^{y} \Phi_i.
	\end{equation}
	Hence in shuffling $R_x^y$ past $R_z^x$ we have a new word for $\Phi$ and in it we have introduced a new letter $R_z^y$ that is in $H$.

	Now we consider when $\Phi_i = R_x^y$ and $\Phi_{i+1} = \pi^{v^{\pm1}}_C$.
	By the choice of $\Phi_i$ we know $o(v)\neq 2$.
	Hence, by Lemma~\ref{lem:pc and transv not commute gives SIL}, if $\Phi_i$ does not commute with $\Phi_{i+1}$ then we must have $v=x$ and 
	\begin{equation}\label{eq:shuffle 2}
	\Phi_i \Phi_{i+1}
	= \pi^y_C \Phi_{i+1} \Phi_i
	= \Phi_{i+1} \pi^y_C \Phi_i
	\end{equation}
	since otherwise $(v,y\mids x)$ or $(v,y\mids C)$ would be a non-Coxeter SIL.
	As above, the letter introduced in this process, namely $\pi^y_C$, is in $H$.
	
	If, on the other hand, $\Phi_i = \pi_C^v$ (now $o(v) = 2$),
	then it follows from Lemma~\ref{lem:pc and transv not commute gives SIL} and Lemma~\ref{lem:partial conj not commute} that $\Phi_i$ commutes with $\Phi_{i+1}\not\in H$ (recall that the domain of $\mu$ is $\Out^1_\infty(\G)$, and no transvection in $\ker\mu$ can act on $v$).

	To see that we can shuffle all the letters from $H$ to the end, group together strings of letters in $w$ according to whether they are in $H$ or not. So we will have $H$--syllables (maximal length subwords consisting of generators with multipliers of order 2) and non-$H$--syllables (maximal length subwords consisting of of generators with multipliers of order not 2).
	Write $w = \alpha_1 h_1 \alpha_2 h_2 \cdots \alpha_lh_l$, where $\alpha_i$ are non-$H$--syllables and $h_i$ are $H$--syllables.
	Then, using equations~\eqref{eq:shuffle 1} and~\eqref{eq:shuffle 2}, shuffling $h_{i}$ past $a_{i+1}$ will give us a new word $w' =\alpha_1h_1 \cdots \alpha_i \alpha_{i+1} h'\alpha_{i+2} \cdots \alpha_lh_l$ for some $h'\in H$.
	The number of non-$H$--syllables in $w'$ has decreased by one.
	Repeating the shuffling in this way we see that ultimately we will have the form $\Phi = \alpha h$ as required.

Let $H_1\le H$ be the subgroup generated by automorphisms of type (\ref{gen:cpc}) and (\ref{gen:ctr}) (i.e.{} those whose multiplier is a commutator). We claim that $H_1=H'$,  the commutator subgroup of $H$ and is thus finite index in $H$ (since $H/H'$ is finitely generated by torsion elements and is abelian, hence is finite).
This claim completes the proof, since then, given $T=\{t_1,\ldots, t_k\}$, a complete set of right-coset representatives for $H_1\le H$, any $\Phi \in \ker \mu$ can be written as $\Phi = \alpha  \alpha_1 t_i$ for $\alpha\in\cal{A}$ as above, $\alpha_1 \in \cal{A}\cap H_1$, and some $i$.

To prove the claim, that $H_1 = H'$, first note that $H_1 < H'$ since each generator of $H_1$ is a commutator of generators of $H$.
To get the other inclusion, using the identity $[ab,c] = a[b,c]a\m[a,c]$, it is enough to show that $H_1$ is normal in $H$ and commutators of generators of $H$ are in $H_1$.

We first show that $H_1$ is normal in $H$. 
 Consider conjugates of $R=R^{[y,z]}_x$, where $o(y)=o(z)=2$.
First, conjugating $R$ by a partial conjugation will keep us in $H_1$ by Lemma~\ref{lem:pc and comm transv not commute STIL}.
	If we conjugate $R$ by a transvection with order 2 multiplier, then the result, by Lemma~\ref{lem:no SIL comm transv commute with transv}, is a product of commutator transvections (with commutators $[y,z]$ or its inverse) and conjugates of these by partial a conjugation in $H$.

Now consider conjugates of $\pi=\pi^{[v,w]}_C$, with $o(v)=o(w)=2$. 
In particular, for $\pi$ to be non-trivial we require a SIL $(v,w\mids z)$ for some $z\in C$.
Consider the conjugate $R^y_x \pi R^y_x$, where $o(y)=2$ and $o(x)=\infty$.
Assuming $R_x^y$ and $\pi$ do not commute, without loss of generality, we may assume $R_x^y$ and $\pi^v_C$ do not commute.
Lemma~\ref{lem:pc and transv not commute gives SIL} then tells us $(v,y\mids x)$ is a SIL and $\lk(x)\subseteq \lk(v)\cap\lk(y)$ (the other case is not possible since $o(v)\neq o(x)$).
Also, from Lemma~\ref{lem:commuting criteria}, we must have $x\in C$.
In particular we therefore have SILs $(v,y\mids x)$ and $(v,w\mids x)$.
These are sufficient to imply that $(v,w,y\mids x)$ is a STIL by Lemma~\ref{lem:overlapping SILs give STIL}, a contradiction.

The case where $\pi^{[v,w]}_C$ is conjugated by a partial conjugation $\pi^y_D$ with $o(y)=2$ is resolved by Lemma~\ref{lem:SS17}. 

Thus, $H_1$ is normal in $H$.

By definition, any commutator of two transvections or two partial conjugations in $H$ is contained in $H_1$. 
It is therefore enough to show that a commutator of the form $[R^y_x, \pi^v_C]$ lies in $H_1$ for $x\le_\infty y$, $o(y)=o(v)=2$. However, since $x\neq v$, by Lemma~\ref{lem:pc and transv not commute gives SIL} either $(v,y\mids x)$ is a SIL or the commutator is trivial. Without loss of generality, we can assume that $y\not\in C$. If $x\not\in C$, the commutator is trivial. If $x\in C$, then $x\le_\infty v$ since $\lk(x) \subseteq \lk(v)\cap\lk(y)$ by Lemma~\ref{lem:pc and transv not commute gives SIL}, and direct computation shows that $[R^y_x, \pi^v_C]=R^{[y,v]}_x$.
Thus, $H_1=H'$.
\end{proof}

Before completing the proof of Theorem~\ref{thm:no SIL SES}, we look at the relations between the generators of $\cal A$.

\begin{prop}\label{prop:calA commutators}
	Suppose that $(\Gamma, o)$ satisfies the hypotheses for Theorem~\ref{thm:no SIL SES}. Then, all the generators of types \eqref{gen:pc}, \eqref{gen:cpc}, \eqref{gen:tr}, \eqref{gen:ctr} commute with one-another, with the following exceptions:
	\begin{enumerate}
		\renewcommand{\theenumi}{$\cal{C}$\arabic{enumi}}
		\item\label{comm:pc and tr} $[R_x^y, \pi^x_C] = \pi^{y}_C$,
		\item\label{comm:tr and tr} $[R_x^{y} , R_w^x] = R_w^{y}$.
	\end{enumerate}
\end{prop}

\begin{proof}
	We consider $[\alpha,\beta]$, with $\alpha,\beta$ coming from automorphisms of types \eqref{gen:pc}, \eqref{gen:cpc}, \eqref{gen:tr}, \eqref{gen:ctr}.
	
	\eqref{gen:pc} and \eqref{gen:pc}.
	These commute by Lemma~\ref{lem:partial conj not commute} since $\Gamma$ has no non-Coxeter SILs.
	
	\eqref{gen:pc} and \eqref{gen:cpc}.
	These commute by Lemma~\ref{lem:SS17}.
	
	\eqref{gen:pc} and \eqref{gen:tr}.
	Consider $[\pi^v_C , R_x^y]$, with $o(v),o(y) \neq 2$.
	By Lemma~\ref{lem:pc and transv not commute gives SIL}, either the commutator is the identity, or
	we get a non-Coxeter SIL $(v,x\mids y)$ in $\Gamma$, a contradiction,
	or we have  $[R_x^y, \pi^x_C] = \pi^{y}_C$, which is relation~\eqref{comm:pc and tr}.

	\eqref{gen:pc} and \eqref{gen:ctr}.
	Consider $[\pi^v_C , R_x^{[y,z]}]$.
	By Lemma~\ref{lem:pc and comm transv not commute STIL}, the lack of a STIL or FSIL means this commutator is the identity 
	else $v\in \{y,z\}$. But $o(v)\neq 2$ and $o(y)=o(z)=2$, so the commutator is indeed the identity.
	
	\eqref{gen:cpc} and \eqref{gen:cpc}.
	These commute by Lemma~\ref{lem:SS17}.

	\eqref{gen:cpc} and \eqref{gen:tr}.
	Consider $[\pi_C^{[v,w]} , R_x^y]$, with $o(y)\neq 2$, and $o(v) = o(w) =2$.
	By virtue of their orders, the vertices $v,w,x,y$ are distinct.
	Lemma~\ref{lem:pc and transv not commute gives SIL} then implies that $R_x^y$ commutes with both $\pi^v_C$ and $\pi^w_C$, since otherwise we will have a non-Coxeter SIL.
	Thus $[\pi_C^{[v,w]} , R_x^y]=1$.

	\eqref{gen:cpc} and \eqref{gen:ctr}.
	Consider $[\pi_C^{[v,w]} , R_x^{[y,z]}]$.
	By Lemma~\ref{lem:pc and comm transv not commute STIL}, if the commutator is to be non-trivial, since there is no STIL or FSIL, we have that
	 $C = \{x\}$ and $\{v,w\}\cap \{y,z\} \ne \emptyset$.
	Without loss of generality, assume $v=y$.
	Then we have SILs $(y,z \mids x)$ and $(y,w\mids x)$. If $w\ne z$, Lemma~\ref{lem:overlapping SILs give STIL} implies $(y,z,w\mids x)$ is a STIL.
	Hence we may assume $v=y$ and $w=z$, and direct calculations show that the automorphisms then commute.

	\eqref{gen:tr} and \eqref{gen:tr}.
	Lemma~\ref{lem:relators} resolves this case, in particular we get \eqref{comm:tr and tr}.

	\eqref{gen:tr} and \eqref{gen:ctr}.
	Consider $[R_v^w,R_x^{[y,z]}]$, with $o(w) \neq 2$.
	On account of their orders, $w,y,z$ are distinct. Lemma~\ref{lem:no SIL comm transv commute with transv} then tells us that the commutator is the identity.

	\eqref{gen:ctr} and \eqref{gen:ctr}.
	Consider $[R_u^{[v,w]},R_x^{[y,z]}]$.
	If $u\neq x$ then by Lemma~\ref{lem:relators} the commutator transvections commute.
	Suppose $u=x$. If $\{v,w\} \cap \{y,z\}$ is empty, then Lemma~\ref{lem:no SIL comm transv commute with transv} implies they again commute.
	If $\{v,w\} = \{y,z\}$, then the commutator transvections are equal, up to taking inverses, so commute.
	So we may assume $v=y$ and $w\neq z$.
	We have two SILs, $(v,w\mids u)$ and $(v,z \mids u)$, so Lemma~\ref{lem:overlapping SILs give STIL} implies that $(v,w,z\mids u)$ is a STIL.
\end{proof}

\begin{proof}[Proof of Theorem~\ref{thm:no SIL SES}]
	By Proposition~\ref{prop:calA finite index} the subgroup $\cal{A}$ has finite index in $\ker \mu$.
	To complete the proof, we show that $\cal{A}$ is nilpotent, of the desired nilpotency class.

	Let $S_i$ be the set consisting of all partial conjugations $\pi^v_C$ so that $o(v)\ne 2$ and $v$ has $\infty$--depth at least $i$, and all transvections $R_u^v$ so that $o(v) \ne 2$ and the difference in $\infty$--depths of $u$ and $v$ is at least $i$.
	To $S_1$ we also add the commutator partial conjugations and commutator transvections of types \eqref{gen:cpc} and \eqref{gen:ctr}.
	Then $\grp{S_1} = \cal{A}$.
	Let $c$ be the $\infty$--depth of $(\Gamma,o)$.
	Then $S_{i}$ is empty for every $i>c$.

	Proposition~\ref{prop:calA commutators} tells us that if $\alpha \in S_i$ and $\beta \in S_j$ then $[\alpha,\beta] \in \grp{S_{i+j}}$ (the empty set generates the trivial group).	
	This implies that $\cal{A}$ is nilpotent with nilpotency class at most $c$.
	
	To see that it has nilpotency class exactly $c$, 
	first suppose $c=1$. 
	Then the relations \eqref{comm:pc and tr} \eqref{comm:tr and tr} do not occur in $\cal{A}$, so all generators commute and $\cal{A}$ is abelian.
	
	Now suppose $c>1$.
	Consider a vertex $v$ of maximal $\infty$--depth $c$.
	Then either \eqref{depth:tr only} holds, so there is a chain of vertices 
	$$v_1 \leq_\infty v_2 \leq_\infty \cdots \leq_\infty v_c = v$$
	or \eqref{depth:tr and pc} holds and there is a chain 
	$$v_1 \leq_\infty v_2 \leq_\infty \cdots \leq_\infty v_{c-1} = v$$
	where $\Gamma \setminus \st(v_1)$ contains two components $C_1,C_2 \not\subseteq \st(v)$.
	If \eqref{depth:tr only} holds we have
	$$\left[\cdots \left[\left[R^v_{v_{c-1}} , R^{v_{c-1}}_{v_{c-2}} \right] , R^{v_{c-2}}_{v_{c-3}} \right] , \cdots R_{v_1}^{v_2} \right] = R_{v_1}^v$$
	while for \eqref{depth:tr and pc} we instead have
	$$\left[R^v_{v_{c-2}} , \cdots \left[ R^{v_3}_{v_2} , \left[ R^{v_2}_{v_{1}} , \pi^{v_{1}}_{C} \right]  \right]  \cdots \right] = \pi_{C}^v.$$
	In both cases this  is a non-trivial element in the center of $\cal{A}$, giving the required lower bound on its nilpotency class.
 \end{proof}

\begin{proof}[Proof of Theorem~\ref{thm:no SIL nil class}]
	 The existence of a free subgroup was described above, immediately after the statement of the theorem. So we assume $(\Gamma,o)$ has none of the given graphical features.
	
	We consider the subgroup $\Out^1_\infty(\G)$, which is finite index in $\Out(\G)$ by Proposition~\ref{prop:kernel D}. Let $c$ be the $\infty$--depth of $(\Gamma,o)$.
	
	Since $(\Gamma, o)$ contains no non-Coxeter SILs, STILs or FSILs, then by Theorem~\ref{thm:no SIL SES} there is a short exact sequence:
	$$1 \to P \to \Out^1_\infty(\G_\Gamma) \to \SL(n_i, \Z)\to 1,$$
	where $P$ is finitely generated virtually nilpotent of nilpotency class $c$, and $(n_1, \ldots n_k)$ are the sizes of the infinite--order $\infty$--equivalence classes of $\Gamma$. 
	
	By assumption, $n_i=1$ for all $i$, and hence $P\cong \Out^1_\infty(G_\Gamma)$, and hence by Theorem~\ref{thm:no SIL SES}, $\Out^1_\infty(\G_\Gamma)$ is virtually nilpotent of class $c$.
\end{proof}

\bibliographystyle{alpha}
\bibliography{bibliography_out_gp_ta1}

\medskip
\noindent {Andrew Sale}, 
University of Hawaii at Manoa, \\
{andrew@math.hawaii.edu}, \ \href{https://math.hawaii.edu/~andrew/}{https://math.hawaii.edu/$\sim$andrew/}
\medskip

\noindent {Tim Susse}, 
Bard College at Simon’s Rock, \\
 {tsusse@simons-rock.edu}, \ \href{https://sites.google.com/site/tisusse/}{https://sites.google.com/site/tisusse/}

\end{document}